\documentclass[12pt,reqno]{amsart}
\usepackage{amssymb,amsmath,amsthm,enumerate,bbm}
\usepackage[a4paper]{geometry}

\DeclareMathOperator{\sign}{sign}
\DeclareMathOperator{\Ran}{Ran}

\DeclareMathOperator{\Ker}{Ker}
\DeclareMathOperator{\Tr}{Tr}

\DeclareMathOperator{\rank}{rank}
\DeclareMathOperator{\supp}{supp}

\DeclareMathOperator*{\slim}{s-lim}

\DeclareMathOperator{\SHO}{SHO}
\DeclareMathOperator{\const}{const}

\renewcommand\Im{\hbox{{\rm Im}}\,}
\renewcommand\Re{\hbox{{\rm Re}}\,}
\newcommand{\abs}[1]{\lvert#1\rvert}
\newcommand{\Abs}[1]{\left\lvert#1\right\rvert}
\newcommand{\norm}[1]{\lVert#1\rVert}

\newcommand{\bbT}{{\mathbb T}}
\newcommand{\bbR}{{\mathbb R}}
\newcommand{\bbC}{{\mathbb C}}

\newcommand{\bbN}{{\mathbb N}}

\newcommand{\bbS}{{\mathbb S}}

\newcommand{\wh}{\widehat}

\newcommand{\calH}{{\mathcal H}}

\newcommand{\calK}{{\mathcal K}}
\newcommand{\calF}{\mathcal{F}}
\newcommand{\calE}{\mathcal{E}}
\newcommand{\calL}{\mathcal{L}}

\newcommand{\calN}{\mathcal{N}}

\newcommand{\Sch}{\mathbf{S}}

\newcommand{\bZ}{{\mathbf{Z}}}
\newcommand{\bY}{{\mathbf{Y}}}
\newcommand{\bpsi}{{\pmb{\psi}}}
\newcommand{\bchi}{{\pmb{\chi}}}
\newcommand{\bzeta}{{\pmb{\zeta}}}
\newcommand{\bOmega}{{\pmb{\Omega}}}
\newcommand{\bomega}{{\pmb{\omega}}}
\newcommand{\fh}{{\mathfrak{h}}}
\newcommand{\f}{\varphi}

\numberwithin{equation}{section}

\theoremstyle{plain}
\newtheorem{theorem}{\bf Theorem}[section]
\newtheorem*{theorem*}{Theorem 1.1$'$}
\newtheorem{lemma}[theorem]{\bf Lemma}
\newtheorem{proposition}[theorem]{\bf Proposition}
\newtheorem{assumption}[theorem]{\bf Assumption}

\theoremstyle{definition}

\theoremstyle{remark}
\newtheorem*{remark*}{\bf Remark}

\newcommand{\wt}{\widetilde}
\newcommand{\eps}{\varepsilon}

\newcommand{\ac}{\mathrm{ac}}
\newcommand{\1}{\mathbbm{1}}

\begin{document}

\title[Spectral density of difference of spectral projections]{The spectral density of a difference of spectral projections}

\author{Alexander Pushnitski}
\address{Department of Mathematics, King's College London, Strand, London, WC2R~2LS, U.K.}
\email{alexander.pushnitski@kcl.ac.uk}



\begin{abstract} 
Let $H_0$ and $H$ be a pair of self-adjoint operators 
satisfying some standard assumptions of scattering
theory. 
It is known from previous work that if $\lambda$ belongs to the absolutely continuous spectrum of 
$H_0$ and $H$, then the difference of spectral projections
$$
D(\lambda)=\1_{(-\infty,0)}(H-\lambda)-\1_{(-\infty,0)}(H_0-\lambda)
$$
in general is not compact and has non-trivial absolutely continuous spectrum. 
In this paper we consider the compact approximations 
$D_\varepsilon(\lambda)$ of $D(\lambda)$,
given by 
$$
D_\varepsilon(\lambda)
=
\psi_\varepsilon(H-\lambda)-\psi_\varepsilon(H_0-\lambda),
$$
where $\psi_\varepsilon(x)=\psi(x/\varepsilon)$ and 
$\psi(x)$ is a smooth real-valued function which 
tends to $\mp1/2$ as $x\to\pm\infty$. 
We prove that the eigenvalues of $D_\varepsilon(\lambda)$ concentrate 
to the absolutely continuous spectrum of $D(\lambda)$ as $\eps\to+0$. 
We show that the rate of concentration is proportional to $|\log\varepsilon|$ and 
give an explicit formula for the asymptotic density of these eigenvalues.
It turns out that this density is independent of $\psi$.
The proof relies on the analysis of Hankel operators. 
\end{abstract}

\maketitle

\section{Introduction}\label{sec.a}

\subsection{Background}
Let $H_0$ and $H$ be self-adjoint operators in a Hilbert space 
such that the difference $H-H_0$ is compact.
Then it is not difficult to show that for any continuous function 
$\f:\bbR\to\bbR$ which tends to zero at infinity, the difference
\begin{equation}
\f(H)-\f(H_0)
\label{a0}
\end{equation}
is also compact. 
However, if $\f$ has discontinuities on the essential spectrum 
of $H_0$ and $H$, then the difference \eqref{a0} may acquire 
non-trivial absolutely continuous (a.c.) spectrum. The first example 
of this kind was constructed by M.~G.~Krein in \cite{Krein}.
He was interested in the difference of spectral projections 
\begin{equation}
D(\lambda)
=
\1_{(-\infty,0)}(H-\lambda)-\1_{(-\infty,0)}(H_0-\lambda);
\label{a3}
\end{equation}
here $\1_{(-\infty,0)}$ is the characteristic function of the 
interval $(-\infty,0)$. 
Krein exhibited an explicit pair of bounded operators $H_0$, $H$
with $\rank(H-H_0)=1$; for $\lambda$ in the a.c. spectrum of $H_0$, 
he computed the difference \eqref{a3} and showed that it is 
not in the Hilbert-Schmidt class, which sufficed for his purposes. 
Later in \cite{KM}, using methods of the theory of Hankel operators, 
it was shown that the operator $D(\lambda)$  in Krein's example 
is not even compact and has non-trivial a.c. spectrum.

In \cite{Push1,PY1,PY2} this phenomenon was studied in a general setting. 
Let us briefly recall the results of this work. Suppose that 
the pair $H_0$, $H$ satisfies some
standard assumptions of smooth scattering theory.
It was shown that for $\lambda$ in the a.c. spectrum of $H_0$ and $H$, 
the spectral structure of $D(\lambda)$ can be described in terms of 
the scattering matrix $S(\lambda)$ for the pair $(H_0,H)$. 
The scattering matrix is a unitary operator in an auxiliary 
Hilbert space $\calN$, which is a fiber space in the spectral representation of $H_0$. 
The space $\calN$ may be finite or infinite dimensional; we denote $N=\dim\calN\leq\infty$. 
Let $\{e^{i\theta_n(\lambda)}\}_{n=1}^N$
be the eigenvalues of $S(\lambda)$ and let us denote
$$
a_n(\lambda)
=
\frac12\abs{e^{i\theta_n(\lambda)}-1}
=
\abs{\sin(\theta_n(\lambda)/2)}
$$
for $\lambda$ in the a.c. spectrum of $H_0$. 
In \cite{Push1,PY1} it was proven that the a.c. spectrum of $D(\lambda)$
can be characterised as the union of the intervals
\begin{equation}
\sigma_\ac(D(\lambda))
=
\bigcup_{n=1}^N [-a_n(\lambda),a_n(\lambda)],
\label{a5}
\end{equation}
where each interval contributes multiplicity one to the a.c. spectrum.
It was also proven that the singular continuous spectrum of $D(\lambda)$ 
is absent and some partial information about the eigenvalues of $D(\lambda)$
was obtained. Finally, in \cite{PY2}, the spectrum of $\f(H)-\f(H_0)$ was studied
for functions $\f$ with several jump discontinuities.

\subsection{Informal description of the main result}

In this paper, we study the regularizations of the difference $D(\lambda)$ 
obtained by replacing the characteristic function $\1_{(-\infty,0)}$
in the definition \eqref{a3} by a smooth function $\psi_\eps$ which 
approaches $\1_{(-\infty,0)}$ as $\eps\to+0$.
More precisely, let $\psi\in C^\infty(\bbR)$ be a real valued function such that 
\begin{equation}
\psi(x)
\to
\begin{cases}
1/2 & \text{ as } x\to-\infty,
\\
-1/2 & \text{ as } x\to\infty.
\end{cases}
\label{a6}
\end{equation}
For $\eps>0$, we denote $\psi_\eps(x)=\psi(x/\eps)$ and consider the difference 
\begin{equation}
D_\eps(\lambda)
=
\psi_\eps(H-\lambda)
-
\psi_\eps(H_0-\lambda).
\label{a7}
\end{equation}
Clearly, we have
$$
\psi_\eps(x)\to\1_{(-\infty,0)}(x)-\tfrac12, 
\quad \eps\to+0, \quad x\not=0
$$
and therefore, if $\lambda$ is not an eigenvalue of $H_0$ or $H$, 
then 
$$
D_\eps(\lambda)\to D(\lambda) 
\quad\text{ strongly as $\eps\to+0$.}
$$
Fix $\lambda$ in the a.c. spectrum of $H_0$. 
By the results of \cite{Push1,PY1}, the a.c. spectrum of $D(\lambda)$ 
is described by the union of the bands \eqref{a5}. 
On the other hand, under our assumptions, the operator $D_\eps(\lambda)$ 
is compact (see Lemma~\ref{lma.b0}) and so it has pure point spectrum.
One expects that the eigenvalues of $D_\eps(\lambda)$ 
concentrate to the spectral bands \eqref{a5} as $\eps\to+0$. 
We show that this is indeed the case
and give a quantitative description of this concentration.

Let us briefly and somewhat informally describe our assumptions
on the self-adjoint operators $H_0$ and $H$; precise statements
will be given in Section~\ref{sec.b}.
We assume that $H_0$ is lower semi-bounded and $H=H_0+V$, 
where $V$ is $H_0$-form compact. 
(Lower semi-boundedness is not essential for our construction but it allows us to 
avoid some unimportant technical issues.) 
We assume that for some
$b\in\bbR$ and for some $k\in\bbN$
$$
(H_0+bI)^{-k}-(H+bI)^{-k}\in\Sch_p,
$$
where $\Sch_p$ is a Schatten class with an exponent $p<\infty$. 
Finally, and most importantly, we make (in the terminology of \cite{Yafaev1})
a strong smoothness assumption. Let us fix an open interval $\delta\subset \bbR$; 
in what follows, the parameter $\lambda$ will be taken inside this interval.  
Roughly speaking, the strong smoothness assumption
means that 
\begin{enumerate}[(i)] 
\item
$H_0$ has a purely a.c. spectrum of a constant multiplicity $N\leq\infty$ on $\delta$;
\item
the operator $\1_\delta(H_0)V\1_\delta(H_0)$ can be 
represented as an integral operator with a sufficiently regular kernel
in the spectral representation of $H_0$. 
\end{enumerate}

\emph{Our main result is as follows:}
Let $g\in C(\bbR)$ be a function that vanishes identically 
in a neighbourhood of zero. Then for every $\eps>0$, the operator $g(D_\eps(\lambda))$
has a finite rank and so its trace is well defined;
we prove that for any $\lambda\in\delta$ one has the asymptotic relation
\begin{equation}
\Tr g(D_\eps(\lambda))
=
\abs{\log\eps} 
\int_{-1}^1 g(y)\mu_\lambda(y)dy
+
o(\abs{\log\eps}), 
\quad
\eps\to+0,
\label{a8}
\end{equation}
where the density function $\mu_\lambda$ is given by 
\begin{equation}
\mu_\lambda(y)
=
\frac1{\pi^2}\sum_{n=1}^N
\frac{\1_{(-a_n(\lambda),a_n(\lambda))}(y)}{\abs{y}\sqrt{1-y^2/a_n^2(\lambda)}},
\quad
y\in(-1,1).
\label{a9}
\end{equation}
\subsection{Discussion}

\emph{Universality:}
Observe that $\mu_\lambda$ is \emph{independent of the choice of $\psi$},
as long as $\psi$ satisfies \eqref{a6}. 
Further,
the density $\mu_\lambda$ is the sum of the functions
each of which is supported on a single band $[-a_n(\lambda), a_n(\lambda)]$.
Each of these functions is a scaled version of the explicit function
$$
\frac1{\pi^2}
\frac{\1_{(-1,1)}(y)}{\abs{y}\sqrt{1-y^2}}, 
\quad
y\in(-1,1).
$$
This can be interpreted as a certain universality phenomenon in this spectral problem. 

We also note that 
by shifts and scaling, it is easy to obtain analogous 
results in the case when the
function $\psi\in C^\infty(\bbR)$ satisfies 
\begin{equation}
\psi(x)\to A_\pm\quad \text{ as $x\to\pm\infty$,}
\label{a9d}
\end{equation} 
for any given values $A_+\not=A_-$.

\emph{Symmetry of $\mu_\lambda$:}
Observe that $\mu_\lambda$ is even, $\mu_\lambda(-y)=\mu_\lambda(y)$.
In particular, \eqref{a8} yields 
\begin{equation}
\Tr g(D_\eps(\lambda))
=
o(\abs{\log\eps}), 
\quad 
\eps\to+0, 
\quad 
\text{ $g$ odd.}
\label{a10}
\end{equation}
We shall give some explanation of this in Section~\ref{sec.c2}.

\emph{Logarithmic rate:}
Let us present a heuristic argument that provides some intuition into the
appearance of the  logarithmic term $\abs{\log\eps}$ in \eqref{a8}.
We use the formalism of the double operator integrals, see e.g. \cite{BSdoi} and references therein.
Fix $\lambda\in\delta$; we have
\begin{equation}
\psi_\eps(H-\lambda)-\psi_\eps(H_0-\lambda)
=
\int_\bbR\int_\bbR
\frac{\psi_\eps(x-\lambda)-\psi_\eps(y-\lambda)}{x-y}d\calE(x,y),
\label{a9a}
\end{equation}
where $\calE(x,y)$ is the operator valued measure on $\bbR\times\bbR$
given by 
$$
\calE(\Delta,\Delta_0)
=
\1_\Delta(H)V\1_{\Delta_0}(H_0), \quad \Delta,\Delta_0\subset\bbR.
$$
Roughly speaking, our strong smoothness assumption on $V$ ensures that 
the measure $\calE$ is sufficiently regular on  $\delta_0\times\delta_0$,
where $\delta_0\subset \delta$ is an open set which contains $\lambda$. 
``Sufficiently regular'' in this context means that $d\calE(x,y)=\calE'(x,y)dxdy$,
where the norm $\norm{\calE'(x,y)}_p$ in an appropriate Schatten class $\Sch_p$
is bounded uniformly in $(x,y)\in\delta_0\times\delta_0$.

From the regularity of $\calE$
it follows that the singular behaviour of the operator \eqref{a9a} as $\eps\to+0$
is determined entirely by the behaviour of the function 
$\frac{\psi_\eps(x-\lambda)-\psi_\eps(y-\lambda)}{x-y}$ near $x=y=\lambda$ (and not by the measure $\calE$). 
To see why $\abs{\log \eps}$ appears in the asymptotics \eqref{a8}, let us 
compute the Hilbert-Schmidt norm of the operator in \eqref{a9a}:
\begin{multline*}
\Tr(D_\eps(0))^2
=
\norm{\psi_\eps(H-\lambda)-\psi_\eps(H_0-\lambda)}_2^2
\\
=
\int_{\delta_0}\int_{\delta_0}
\Abs{\frac{\psi_\eps(x-\lambda)-\psi_\eps(y-\lambda)}{x-y}}^2\norm{\calE'(x,y)}_2^2 \, dx dy+O(1)
\end{multline*}
as $\eps\to+0$. 
Assuming that  
$\norm{\calE'(x,y)}_2$ is bounded uniformly in $(x,y)\in\delta_0\times\delta_0$, 
we end up with estimating the integral 
$$
\int_{\delta_0}\int_{\delta_0}
\Abs{\frac{\psi_\eps(x-\lambda)-\psi_\eps(y-\lambda)}{x-y}}^2 dx dy.
$$
An elementary calculation (using our assumption \eqref{a6} on the asymptotic behaviour of $\psi$)
shows that this integral has the asymptotics $2\abs{\log \eps}+O(1)$ as $\eps\to+0$. 

\emph{Comparison with other estimates:}
Under somewhat more restrictive assumptions on $\psi$ (see \eqref{a6aa}), 
the asymptotics \eqref{a8} is valid for $g(t)=t^m$, where $m$ is a sufficiently 
large integer (this follows from the first step of the proof in Section~\ref{sec.j}).
Taking $m$ even, directly from \eqref{a8} we obtain an estimate in the Schatten class $\Sch_m$:
\begin{equation}
\norm{\psi_\eps(H-\lambda)-\psi_\eps(H_0-\lambda)}_m^m
=
\Tr (D_\eps(\lambda))^m
=
O(\abs{\log\eps}), 
\quad \eps\to+0.
\label{a9b}
\end{equation}
On the other hand, one can find estimates of the type
\begin{equation}
\norm{\f(H)-\f(H_0)}_m
\leq C\norm{\f}_* \norm{H-H_0}_m
\label{a9c}
\end{equation}
in the literature;
here $\norm{\f}_*$ is the norm of $\f$ in an appropriate function class. 
There are several variations of \eqref{a9c}: with a different Schatten 
norm on the right-hand side (r.h.s.), with a different power of $\norm{H-H_0}$ in the r.h.s., etc; see e.g. \cite{AP,PS}
and references therein.
Substituting $\f(x)=\psi_\eps(x-\lambda)$ into any of the estimates of the type \eqref{a9c}, one \emph{does not}
recover the logarithmic behaviour \eqref{a9b}.
In fact, the best one can get in this way is $O(\eps^{-\alpha})$ with $\alpha>0$. 
This is not surprising, because estimates of the type \eqref{a9c}
are valid for \emph{all} pairs of operators $H_0$, $H$, whereas we use the 
crucial strong smoothness assumption. 
This aspect is further illustrated in the following example.

\emph{Example:}
To show what can happen with $\Tr g(D_\eps(\lambda))$ without any structural
assumptions on $H_0$ and $H$, let us consider the following example. 
Let $H_0=0$ and let $H\in\Sch_p$ be a compact self-adjoint operator with the eigenvalues
$\{\lambda_n\}_{n=1}^\infty$.
Assume $\psi(0)=0$; then 
$$
\Tr g(D_\eps(0))
=
\Tr g(\psi(H/\eps))
=
\sum_{n=1}^\infty g(\psi(\lambda_n/\eps)).
$$
Suppose that $\psi(t)=-1/2$ for $t\geq 1$ and $g\geq0$, $g(-1/2)=1$. 
Then the r.h.s. can be estimated below as
$$
\sum_{n=1}^\infty g(\psi(\lambda_n/\eps))
\geq 
\#\{n: \lambda_n>\eps\}.
$$
Thus, by choosing the sequence $\{\lambda_n\}_{n=1}^\infty$ appropriately,
we can make the r.h.s. behave as $\eps^{-\alpha}$ with any $\alpha<p$.

\emph{Applications:}
In Section~\ref{sec.x} 
we give some examples of applications of the main result to the Schr\"odinger operator. 

\emph{Connection with the trace formula:}
Suppose that $H-H_0$ is a trace class operator. 
In  \cite{Krein}, Krein proved that 
there exists a real-valued function $\xi\in L^1(\bbR)$ (called the 
spectral shift function) such that
the following trace formula holds true:
$$
\Tr(\f(H)-\f(H_0))
=
\int_{-\infty}^\infty \xi(x)\f'(x)dx, 
$$
for all smooth functions $\f$ of a certain class. 
Taking \emph{formally} $\varphi(x)=\1_{(\infty,0)}(x-\lambda)$ 
and observing that in this case $\varphi'(x)=-\delta(x-\lambda)$,
we obtain the ``naive trace formula''
\begin{equation}
\Tr D(\lambda)=-\xi(\lambda).
\label{a1b}
\end{equation}
Since $D(\lambda)$ in general fails to belong to trace class, 
the naive formula \eqref{a1b} does not make sense as it is. 
However, it remains a source of inspiration in this area and 
several regularisations have been considered in the literature
(see e.g. \cite{Nakamura,FLLS}). 
One regularisation is to take $\varphi(x)=\psi_\eps(x-\lambda)$;
since 
$-\psi_\eps'(x-\lambda)$ converges to the delta 
function $\delta(x-\lambda)$, 
we obtain
\begin{equation}
\Tr D_\eps(\lambda)
=
\int_{-\infty}^\infty \xi(x)\psi_\eps'(x-\lambda)dx
\to
-\xi(\lambda),
\label{a1a}
\end{equation}
if $\lambda$ is a Lebesgue point of $\xi$. 
The main result of this paper is a step towards a better understanding of \eqref{a1a}. 

\emph{A conjecture:}
Taking formally $g(t)=t$ in \eqref{a8} (or in \eqref{a10}) and comparing with \eqref{a1a}
suggests that in the trace class case, 
one can hope to replace the error term $o(\abs{\log\eps})$ in \eqref{a8}
(or in \eqref{a10})
by $\const+o(1)$, where the constant is related to the spectral
shift function.

\emph{Related work:}
Much of our construction uses the ideas of \cite{PY2}. 
However, the nature of the results is quite different: \cite{PY2} 
describes the a.c. spectrum of $D(\lambda)$, whereas here we deal
exclusively with the point spectrum. 

In \cite{Push2}, the spectrum of $D_\eps(\lambda)$ was considered
for functions $\psi(x)$ that tend to zero as $\abs{x}\to\infty$. 
In this case, no spectral concentration for $D_\eps(\lambda)$ occurs. 
Instead, the eigenvalues of $D_\eps(\lambda)$
converge to some ``limiting spectrum'', which is described as the spectrum of a certain  
\emph{compact} model operator. This model operator depends
on the scattering matrix $S(\lambda)$ and also \emph{depends on the choice of $\psi$.}
Thus, in this case the universality phenomenon discussed above does not hold.  
To comment on this, we note that the case $\psi(x)\to0$ as $\abs{x}\to\infty$ 
can be considered as a ``degenerate case'' $A_+=A_-=0$ of the function of
the type \eqref{a9d}.
Thus, roughly speaking, the problem discussed in \cite{Push2}
corresponds to the (conjectural) next term in the asymptotics 
 \eqref{a8}.

In \cite{FP} the products of spectral projections 
$$
\Pi_\eps(\lambda)
=
\1_{(-\infty,-\eps)}(H_0-\lambda)
\1_{(\eps,\infty)}(H-\lambda)
\1_{(-\infty,-\eps)}(H_0-\lambda)
$$
are considered under some assumptions similar to the ones of this paper. 
These products are compact, while the limiting product
$$
\Pi(\lambda)
=
\1_{(-\infty,0)}(H_0-\lambda)
\1_{(0,\infty)}(H-\lambda)
\1_{(-\infty,0)}(H_0-\lambda)
$$
in general has a non-trivial a.c. spectrum. 
Similarly to \eqref{a5}, this spectrum can be described as the union of bands
\begin{equation}
\sigma_\ac(\Pi(\lambda))
=
\bigcup_{n=1}^N [0,a_n(\lambda)^2];
\label{a12c}
\end{equation}
this fact was established in \cite{Push1,PY1}. 
In \cite{FP} it is proved that the eigenvalues of $\Pi_\eps(\lambda)$ 
accumulate to the spectral bands \eqref{a12c} in a manner similar to 
\eqref{a8}. The technique of \cite{FP} is quite different from the one of this paper,
although it also relies on the analysis of Hankel operators.

\subsection{Key ideas of the proof}
We study the operator $D_\eps(\lambda)$ in the spectral 
representation of $H_0$. For the simplicity of notation, let us take 
$\lambda=0$ in \eqref{a8} (the general case reduces to this one 
by a shift). 
Let $\chi_0\in C_0^\infty(\bbR)$ be a real valued 
function equal to $1$ in a neighbourhood of zero. Roughly speaking, 
after a number of reductions
we show that the spectrum of $D_\eps(0)$ is accurately approximated
by the spectrum of the compact self-adjoint operator
\begin{equation}
P_-\bpsi_\eps\bchi_0 P_+\otimes (S(0)-I)
+
P_+\bpsi_\eps\bchi_0 P_-\otimes (S(0)^*-I)
\quad \text{ in } L^2(\bbR)\otimes \calN.
\label{a13}
\end{equation}
Here $P_\pm$ are the orthogonal projections onto the Hardy classes
$H^2_\pm(\bbR)\subset L^2(\bbR)$ (see  Section~\ref{sec.a5}),
 $\calN$ is the Hilbert space 
where the scattering matrix $S(0)$ acts, and $\bpsi_\eps$ (resp. $\bchi_0$) 
is the operator of multiplication by the function $\psi_\eps(x)$ (resp. $\chi_0(x)$)
in $L^2(\bbR,dx)$. 
The operators with the structure \eqref{a13} are called \emph{symmetrised
Hankel operators} (SHO) in \cite{PY2}; they were introduced 
in connection with the study of the spectrum of $\f(H)-\f(H_0)$ 
with piecewise continuous $\f$. 

Formula \eqref{a13} already depends on $S(0)$ in an explicit way. 
In order to deal with the part of the operator \eqref{a13} 
acting on $L^2(\bbR)$, 
we show, roughly speaking, that the relevant spectral asymptotics
is independent of the choice
of the function $\psi$, as long as $\psi(x)$ approaches $\mp 1/2$ as $x\to\pm\infty$. 
This allows us to replace $\psi$ by the explicit function
$$
-\tfrac1\pi\tan^{-1}(x).
$$
In this case, we are able to determine the spectral asymptotics by 
reducing the problem to the analysis of a simple explicit Hankel operator. 
See Sections~\ref{sec.c7}, \ref{sec.c8} for the details of the last step.

\subsection{Notation}\label{sec.a5}
We denote by $\Sch_p$, $p\geq1$, the standard Schatten class
and by $\norm{\cdot}_p$ the norm in this class. 
We will frequently use the H\"older inequality for Schatten classes:
\begin{equation}
\norm{XY}_r\leq \norm{X}_p\norm{Y}_q,
\quad
\tfrac1r=\tfrac1p+\tfrac1q.
\label{a14}
\end{equation}
$\mathbf B$ denotes the class of all bounded operators,
$\Sch_\infty$ is the class of all compact operators and $\norm{\cdot}$ 
is the operator norm. 
For a set $\delta\subset \bbR$, $\1_\delta$ denotes the characteristic
function of $\delta$. 
If $X$ is a Banach space, we denote by $L^p(\bbR,X)$, 
$C(\bbR,X)$ etc. the classes of $X$-valued functions on $\bbR$. 

Let $\fh_1$, $\fh_2$ be Hilbert spaces, and let 
$\Omega$ be a function on $\bbR$ with values in the set of bounded operators
acting from $\fh_1$ to $\fh_2$. 
Assume that $\Omega\in L^\infty(\bbR,\mathbf B)$.  
We associate with $\Omega$ the bounded
operator (which will be denoted by the corresponding boldface symbol) 
$$
\bOmega: L^2(\bbR,\fh_1)\to L^2(\bbR,\fh_2)
$$
acting as ``the multiplication by $\Omega$":
\begin{equation}
(\bOmega f)(x)=\Omega(x)f(x), 
\quad
x\in\bbR, 
\quad
f\in L^2(\bbR,\fh_1).
\label{a12}
\end{equation}

We will work with the standard Hardy classes $H^2_\pm(\bbR)\subset L^2(\bbR)$ 
defined as the classes of functions $f\in L^2(\bbR)$ that admit
an analytic continuation into the half-plane 
$\bbC_\pm=\{z\in\bbC: \pm\Im z>0\}$ and satisfy the estimate
$$
\sup_{y>0}\int_{-\infty}^\infty\abs{f(x\pm iy)}^2 dx<\infty.
$$
We denote by $P_\pm: L^2(\bbR)\to H^2_\pm(\bbR)$ the Hardy projections. 
Recall that the explicit formula for $P_\pm$ is
\begin{equation}
(P_\pm f)(x)
=
\mp \frac1{2\pi i} \lim_{\epsilon\to+0}
\int_{-\infty}^\infty
\frac{f(x')}{x-x'\pm i\epsilon}dx', 
\quad 
\text{ a.e. $x\in\bbR$.}
\label{a11}
\end{equation}

\subsection{The structure of the paper}
In Section~\ref{sec.b} we state our assumptions and present 
the main result of the paper (Theorem~\ref{thm.b1}). 
In Section~\ref{sec.c} we describe the plan of the proof, define 
all the main objects appearing in our construction and state 
the main steps of the proof as lemmas. 
These lemmas are proven in Sections~\ref{sec.d}--\ref{sec.i}. 
The proof is concluded in Section~\ref{sec.j}. 
Applications are discussed in Section~\ref{sec.x}.

\section{Main result}\label{sec.b}

\subsection{Assumptions}\label{sec.b1}
Let $H_0$ be a  self-adjoint lower semi-bounded 
operator in a Hilbert space $\calH$. 
Let the perturbation $V$ be of the form
$$
V=G^*V_0G \quad \text{ in $\calH$.}
$$
Here
$G$ is an operator from $\calH$ to
an auxiliary Hilbert space $\calK$
and $V_0=V_0^*$ is a bounded operator in $\calK$
(of course, the simplest case is $\calK=\calH$, $G=\abs{V}^{1/2}$ and $V_0=\sign(V)$).
We assume that the operator $G$ satisfies 
\begin{equation}
G(H_0+bI)^{-1/2}\in \Sch_{\infty}, \quad
b>-\inf \sigma(H_0).
\label{b2}
\end{equation}
Condition \eqref{b2} ensures that $V$ is $H_0$-form compact, 
and so the perturbed operator 
$$
H=H_0+V\quad \text{ in $\calH$.}
$$ 
may be defined as a form sum. 
From \eqref{b2} it is easy to derive
\begin{lemma}\label{lma.b0}
Let \eqref{b2} hold true and 
let $\psi\in C^\infty(\bbR)$ satisfy \eqref{a6}. 
Then for any $\lambda\in\bbR$ and for any $\eps>0$, the operator
$D_\eps(\lambda)$ is compact. 
\end{lemma}
The proof is given in Section~\ref{sec.d}.

Next, we assume that for some
$p<\infty$, some $b>-\min\{\inf\sigma(H_0),\inf\sigma(H)\}$ and some $k\in\bbN$, one has
\begin{equation}
(H+bI)^{-k}-(H_0+bI)^{-k}\in \Sch_{p}.
\label{b2a}
\end{equation}
Further, we describe the strong smoothness assumption. 
Let $\delta\subset \bbR$ be an open interval and let $\overline{\delta}$ be the closure of $\delta$. 
Assume that the spectrum of $H_0$ on $\delta$ is purely a.c. with a 
constant multiplicity $N\leq\infty$. 
More explicitly, we assume that for some auxiliary Hilbert space $\calN$
with $\dim\calN=N$, there exists a unitary operator
$$
\calF:\Ran \1_{\delta}(H_0)\to L^2(\delta,\calN)
$$
such that $\calF$ diagonalises $H_0$:
$$
(\calF H_0 f)(\lambda)=\lambda \wt f(\lambda), \quad 
\lambda\in \delta,
\quad
\wt f=\calF f,
$$
for any $f\in\Ran \1_{\delta}(H_0)$.
Further, for $\lambda\in \delta$, let $Z(\lambda): \calK\to\calN$ 
be the operator formally defined by the relation 
\begin{equation}
Z(\lambda)u=(\calF G^* u)(\lambda), \quad \lambda\in\delta, \quad u\in\calK.
\label{b3a}
\end{equation}
We assume that $Z(\lambda)\in\Sch_{2p}$ and that the estimates 
\begin{equation}
\norm{Z(\lambda)}_{2p}\leq C,
\quad
\norm{Z(\lambda)-Z(\lambda')}_{2p}\leq C\abs{\lambda-\lambda'}^\gamma,
\quad 
\lambda, \lambda'\in \overline{\delta}
\label{b5}
\end{equation}
are satisfied with some exponent $\gamma\in(0,1)$. 
In other words, $Z\in C^\gamma(\overline{\delta},\Sch_{2p})$. 
Formula \eqref{b3a} implies that $G$ acts upon any $f\in\Ran\1_{\delta}(H_0)$
according to the formula
\begin{equation}
Gf = \int_{\delta} Z(\lambda)^*\wt f(\lambda)d\lambda,
\quad 
\wt f=\calF f\in L^2(\delta,\calN). 
\label{b4}
\end{equation}
This is slightly stronger than what is called strong smoothness
in \cite[Section 4.4]{Yafaev1}; strong smoothness refers to the case when the norm in \eqref{b5}
is the operator norm. It is natural to call the above assumption the 
$\Sch_{2p}$-valued strong smoothness. 
In applications, this assumption is easy to verify, see Section~\ref{sec.x}. 

Let us summarize our assumptions:

\begin{assumption}\label{ass1}
\begin{enumerate}[\rm (1)]
\item
$H_0$ is lower semi-bounded and 
$H=H_0+V$, where $V=G^*V_0G$ satisfies \eqref{b2}. 

\item
For some $b>-\min\{\inf\sigma(H_0),\inf\sigma(H)\}$, 
$p<\infty$ and $k\in\bbN$, we have the inclusion \eqref{b2a}.

\item
$H_0$ has a purely a.c. spectrum with multiplicity $N$ on the interval $\delta$. 

\item
$G$ satisfies the $\Sch_{2p}$-valued strong smoothness assumption
\eqref{b3a}, \eqref{b5} on $\delta$.
\end{enumerate}
\end{assumption}

\begin{remark*}
The assumption that $H_0$ and $H$ are lower semi-bounded is not essential
for our construction. 
We choose to impose this assumption here simply because it allows us to avoid
several irrelevant technical issues and to make the exposition more
transparent. 
In \cite{PY2}, related analysis of $D(\lambda)$ is carried out 
without the lower semi-boundedness assumption. 
\end{remark*}

\subsection{Scattering theory}\label{sec.b2}
For $\Im z\not=0$, we set
\begin{equation}
R(z)=(H-z)^{-1}, 
\quad
R_0(z)=(H_0-z)^{-1}, 
\quad
T(z)=GR_0(z)G^*.
\label{b6}
\end{equation}
Since $G$ is not assumed bounded, the precise definition of $T(z)$ is 
$$
T(z)=G(H_0+bI)^{-1/2}\frac{H_0+bI}{H_0-zI}(G(H_0+bI)^{-1/2})^*.
$$
By the assumption \eqref{b2}, it follows that $T(z)$ is compact.
We need the following well-known results, see e.g. \cite[Section~4]{Yafaev1}.

\begin{proposition}\label{pr.b1}
Let Assumption~\ref{ass1} hold. 
Then the operator-valued function $T(z)$   
defined by \eqref{b6}
is H\"older continuous (with the exponent $\gamma$) in the operator norm
for $\Re z\in \delta$, $\Im z\geq 0$; 
in particular, the limits $T(\lambda+i0)$ exist in $\Sch_\infty$  and 
are H\"older continuous in $\lambda\in \delta$.
Let $\delta_*\subset \delta$ be the set of $\lambda$ such that
the equation 
\begin{equation}
f+T(\lambda+ i0)V_0f=0
\label{b10}
\end{equation} 
has a non-trivial solution $f\in\calK$, and let $\delta_0=\delta\setminus\delta_*$.
Then $\delta_*$ 
is closed in $\delta$  and  
has the Lebesgue measure zero.
The inverse operator $(I+T(\lambda+i0)V_0)^{-1}$, $\lambda\in\delta_0$, 
exists, is bounded and is a H\"older continuous function of  
$\lambda\in\delta_0$.
Finally, 
the local wave operators
$$
W_\pm=
W_\pm(H,H_0;\delta_0)
=
\slim_{t\to \pm \infty}e^{iHt}e^{-iH_{0}t}\1_{\delta_0}(H_0)
$$
exist and are complete, i.e.
$\Ran W_\pm=\Ran \1_{\delta_0}(H)$.
\end{proposition}

The local scattering operator 
$$
\mathbf S=W_+^*W_-
$$
is unitary on $\Ran \1_{\delta_0}(H_0)$ and commutes with $H_0$. 
Thus, we have a representation 
$$
(\calF \mathbf S \calF^* u)(\lambda)
=
S(\lambda)u(\lambda), 
\quad \text{ a.e. $\lambda\in \delta_0$, \    $u\in L^2(\delta_0,\calN)$}
$$
where the operator $S(\lambda)$ in $\calN$ is called the 
scattering matrix for the pair $H_0$, $H$. 
The scattering matrix is unitary in $\calN$. 
The difference $S(\lambda)-I$ is a compact operator (in fact, it belongs
to $\Sch_p$, see \eqref{c18}). 
Thus, the spectrum of $S(\lambda)$ consists of eigenvalues 
on the unit circle accumulating possibly only to the point $1$. 
As in Section~\ref{sec.a}, we denote the eigenvalues of $S(\lambda)$ 
by   $\{e^{i\theta_n(\lambda)}\}_{n=1}^N$, $N\leq\infty$, and use the 
notation $a_n(\lambda)=\tfrac12\abs{e^{i\theta_n(\lambda)}-1}$.

\subsection{Main result}\label{sec.b3}

\begin{theorem}\label{thm.b1}
Let Assumption~\ref{ass1} hold true, and let $\lambda\in\delta_0$
(the set $\delta_0$ is defined in Proposition~\ref{pr.b1}). 
Let $\psi\in C^\infty(\bbR)$ satisfy \eqref{a6} and let $D_\eps(\lambda)$ 
be defined by \eqref{a7}. 
Let $g=\1_\omega$, where $\omega\subset\bbR$ is an open interval 
such that $0\notin\overline{\omega}$. 
Then the asymptotic relations
\begin{align}
\lim_{\eps\to+0}\, 
\abs{\log\eps}^{-1}\Tr g(D_\eps(\lambda))
&=
\frac1{2\pi^2}\sum_{n=1}^N 
\int_{-\infty}^\infty 
\biggl\{
g\biggl(\frac{a_n(\lambda)}{\cosh x}\biggr)
+
g\biggl(\frac{-a_n(\lambda)}{\cosh x}\biggr)
\biggr\}
dx
\label{b7}
\\
&=
\int_{-1}^1 g(y)\mu_\lambda(y)dy,
\label{b8}
\end{align}
hold true, 
where $\mu_\lambda$ is given by \eqref{a9}. 
The relations \eqref{b7}, \eqref{b8} also hold true for any 
$g\in C(\bbR)$ that vanishes in a neighbourhood of the origin.
\end{theorem}

Of course, \eqref{b8} follows from \eqref{b7} by a change of variable
$y=a_n/\cosh x$.

\section{The outline of the proof}\label{sec.c}

Here we present the key steps of the proof of Theorem~\ref{thm.b1} and introduce all relevant 
objects. Details are filled in in Sections~\ref{sec.d}--\ref{sec.j}.

\subsection{The strategy}\label{sec.c1}

In what follows for the simplicity of notation we set the parameter $\lambda\in\delta_0$
in Theorem~\ref{thm.b1} to be equal to zero and denote $D_\eps:=D_\eps(0)$. 
We will initially assume that $\psi\in C^\infty(\bbR)$ satisfies a stronger assumption than \eqref{a6}, viz.
\begin{equation}
\psi(x)
=
\begin{cases}
1/2 & \text{ for } x<-R
\\
-1/2 & \text{ for } x>R
\end{cases}
\label{a6aa}
\end{equation}
with some $R>0$. Reduction of the general case to this one is done at the very end
of the proof in Section~\ref{sec.j} by using a simple variational argument.

Our strategy is to transform $D_\eps$ through the sequence of steps
\begin{equation}
D_\eps=:D_\eps^{(0)}
\to D_\eps^{(1)}
\to D_\eps^{(2)}
\to D_\eps^{(3)}
\to D_\eps^{(4)}
\to D_\eps^{(5)}
\label{c1}
\end{equation}
until we arrive at a ``sufficiently simple'' operator $D_\eps^{(5)}$. 
At each step (apart from $D_\eps^{(1)}\to D_\eps^{(2)}$, which is 
a unitary equivalence) we are able to control the error term as follows:
\begin{equation}
\norm{D_\eps^{(i)}-D_\eps^{(i-1)}}_q=O(1), \quad \eps\to+0,
\label{c2}
\end{equation}
where $q$ is a sufficiently large exponent. The precise restrictions on $q$ will
vary with $i$, but the choice $q\geq2p$, $q>1/\gamma$ will work for all $i$; 
here $\gamma$ is the H\"older exponent in the $\Sch_{2p}$ valued strong smoothness 
assumption \eqref{b5}. 

After the estimates \eqref{c2} have been established, the proof proceeds
as follows. The main task is to prove the asymptotics \eqref{b7} for $g(t)=t^m$
for all integers $m\geq q$. (The general case is easily obtained from here by 
a standard application of the Weierstrass approximation theorem.)
Denote by $\Delta_m$ be the r.h.s. 
of \eqref{b7} for $g(t)=t^m$: 
$$
\Delta_m
=
\frac{1+(-1)^m}{2\pi^2}
\sum_{n=1}^N a_n(0)^m
\int_{-\infty}^\infty (\cosh x)^{-m} dx. 
$$
Let $P^{(i)}$ be the statement 
\begin{equation}
P^{(i)}: 
\qquad
\abs{\log\eps}^{-1}\Tr(D_\eps^{(i)})^m\to\Delta_m, 
\quad
\eps\to+0, 
\quad
\forall m\geq q.
\label{c4}
\end{equation}
Our aim is to prove $P^{(0)}$. But we start from the other end of the chain \eqref{c1}:
the operator $D^{(5)}_\eps$ is sufficiently simple so we are able to establish $P^{(5)}$,
see Section~\ref{sec.c8} below. 
After that, using an operator theoretic argument (see Section~\ref{sec.j}), we prove that 
$P^{(i)}$ together with the estimate \eqref{c2} gives $P^{(i-1)}$. 
Thus, moving backwards along the chain \eqref{c1}, we arrive at the 
required statement $P^{(0)}$. 

One exception is the step $D_\eps^{(1)}\to D_\eps^{(2)}$; 
here the operators $D_\eps^{(1)}$ and $D_\eps^{(2)}$
are unitarily equivalent.  Thus,  we have $\Tr(D_\eps^{(1)})^m=\Tr(D_\eps^{(2)})^m$
and so the statements $P^{(1)}$ and $P^{(2)}$ are equivalent. 

The operators $D_\eps^{(i)}$ with $i=3,4,5$ are constructed as operator valued
\emph{symmetrised Hankel operators}. 
In the next subsection we introduce and briefly discuss this concept. In the rest
of this section, we describe each of the operators $D_\eps^{(i)}$ in the chain \eqref{c1}.

\subsection{Symmetrised Hankel operators}\label{sec.c2}

Let $\fh$ be a Hilbert space (the case $\dim\fh<\infty$ is not excluded). 
We will denote by $P_\pm$ the orthogonal
projection in $L^2(\bbR,\fh)$ onto the vector-valued Hardy class $H^2_\pm(\bbR,\fh)$. 
Of course, this projection is given by the same formula \eqref{a11}
as in the scalar-valued case.
For $\Omega\in L^\infty(\bbR,\fh)$ we call 
$$
P_-\bOmega P_+: H^2_+(\bbR,\fh)\to H^2_-(\bbR,\fh)
$$
the \emph{Hankel operator} (HO) with the symbol $\Omega$, 
and we call 
\begin{equation}
\SHO(\Omega)
=
P_-\bOmega P_+
+
P_+\bOmega^* P_-
: L^2(\bbR,\fh)\to L^2(\bbR,\fh)
\label{c6}
\end{equation}
the \emph{symmetrized Hankel operator} (SHO) with the symbol $\Omega$. 
Of course, the notion of  HO is standard, while the notion of SHO is not; 
to the best of the author's knowledge,
SHOs were introduced  in \cite{PY1} as models for the operators of the type $D(\lambda)$.

By definition, SHOs are self-adjoint. 
It is well known (see e.g. \cite[Section 2.4]{Peller}) that 
if the symbol $\Omega\in C(\bbR,\Sch_\infty)$ and $\norm{\Omega(\lambda)}\to0$ as
$\abs{\lambda}\to\infty$, then the corresponding HO (and therefore also the SHO) is compact. 
In this paper, we will only deal with symbols of this class. 

In order to comment on the nature of the spectrum of SHOs, 
we recall without proof a simple operator theoretic statement.

\begin{lemma}\label{lma.c1}
Let $\calH_1$, $\calH_2$ be Hilbert spaces and let $X:\calH_1\to\calH_2$ 
be a compact operator. Then the non-zero eigenvalues of the ``block-matrix''
$$
\begin{pmatrix}
0 & X^*
\\
X & 0
\end{pmatrix}
\quad
\text{ in } 
\calH_1\oplus \calH_2
$$
are given by $\{\pm s_n(X)\}$, where $\{s_n(X)\}$ are the 
non-zero singular values of $X$.
\end{lemma}
The operator $\SHO(\Omega)$ can be written as
$$
\SHO(\Omega)
=
\begin{pmatrix}
0 & (P_-\bOmega P_+)^*
\\
(P_ - \bOmega P_+) & 0
\end{pmatrix}
\quad
\text{ in } 
L^2(\bbR,\fh)=H^2_+(\bbR,\fh)\oplus H^2_-(\bbR,\fh).
$$
Thus, Lemma~\ref{lma.c1} reduces the analysis of the spectrum of 
$\SHO(\Omega)$ to computing the singular values of the 
HO $P_- \bOmega P_+$. We will use this idea below.

Lemma~\ref{lma.c1} also shows that the spectrum of $\SHO(\Omega)$ 
is symmetric with respect to the reflection around zero. 
In particular, if $\Tr g(\SHO(\Omega))$ exists for some odd function $g$, then it equals zero. 
This gives some insight into the symmetry of the density function $\mu_\lambda$.

\subsection{Spectral localization: $D_\eps^{(0)}\to D_\eps^{(1)}$}\label{sec.c3}
Let $\chi_0\in C_0^\infty(\bbR)$ be a real valued function such that 
$\supp\chi_0\subset \delta_0$ and such that $\chi_0(\lambda)=1$ for $\lambda$ 
in a neighbourhood of $\lambda=0$. We set
$$
D_\eps^{(1)}
=
\chi_0(H_0)
(\chi_0(H)^2\psi_\eps(H)-\chi_0(H_0)^2\psi_\eps(H_0))\chi_0(H_0).
$$
Thus,  we regularize $D_\eps^{(0)}$ in 
two ways: (i) we introduce the ``spectral cutoff'' by replacing 
$D_\eps^{(0)}$ with $\chi_0(H_0)D_\eps^{(0)}\chi_0(H_0)$ 
and (ii) we replace $\psi_\eps$ in the definition of $D_\eps^{(0)}$ by
$\psi_\eps\chi_0^2$. (We write $\chi_0^2$ rather than $\chi_0$ here for a trivial 
reason: it will be convenient later to split this term into a product of two cutoff functions.)
Item (i) above is purely technical; item (ii) highlights the fact that the only important aspect
of $\psi_\eps$ is that this function ``approaches a jump'' at zero, and the behaviour
of $\psi_\eps$ outside a neighbourhood of zero is irrelevant. 
We will prove
\begin{lemma}\label{lma.c2}
Let Assumption~\ref{ass1}(1), (2) hold true, and assume \eqref{a6aa}.
Then 
$$
\norm{D_\eps^{(0)}-D_\eps^{(1)}}_p=O(1), \quad \eps\to+0.
$$
\end{lemma}
Note that we do not need the strong smoothness assumption here. 
The proof of the Lemma will be given in Section~\ref{sec.d}; it involves
only some simple $C^\infty$ functional calculus for $H_0$ and $H$.

\subsection{Application of the resolvent identity: $D_\eps^{(1)}\to D_\eps^{(2)}$}\label{sec.c4}
First we need some notation. 
Let 
\begin{equation}
Y(z)=V_0(I+T(z)V_0)^{-1}, \quad \Im z>0
\label{c10}
\end{equation}
(recall that $T(z)$ is defined in \eqref{b6}). 
Let $\chi_0$ be as in the previous subsection.
Using the notation $Z(\lambda)$ (see \eqref{b3a}),
we set
$$
Z_0(\lambda)
=
\begin{cases}
Z(\lambda)\chi_0(\lambda), & \lambda\in \delta_0,
\\
0, & \lambda\in \bbR\setminus\delta_0,
\end{cases}
\quad
Y_0(\lambda)
=
\begin{cases}
Y(\lambda+i0)\chi_0(\lambda), & \lambda\in \delta_0,
\\
0, & \lambda\in \bbR\setminus\delta_0.
\end{cases}
$$
Thus, $Z_0$ and $Y_0$ are operator valued functions on $\bbR$ and by Assumption~\ref{ass1}(4)
and by Proposition~\ref{pr.b1} they are H\"older continuous: 
\begin{equation}
Z_0\in C^\gamma(\bbR,\Sch_{2p}), 
\quad
Y_0\in C^\gamma(\bbR,\mathbf B).
\label{c12}
\end{equation}
We will use the corresponding ``multiplication operators''
$$
\bZ_0: L^2(\bbR,\calK)\to L^2(\bbR,\calN)
\quad \text{ and }\quad
\bY_0: L^2(\bbR,\calK)\to L^2(\bbR,\calK),
$$
defined as in  \eqref{a12}.
Consider the operator 
$$
D_\eps^{(2)}:=\calF D_\eps^{(1)}\calF^* 
\quad\text{ in }\quad
L^2(\delta_0,\calN)\subset L^2(\bbR,\calN).
$$
It will be convenient to consider this operator as an operator acting
on $L^2(\bbR,\calN)$, extending it by zero to $L^2(\bbR\setminus\delta_0,\calN)$. 
Using the resolvent identity for $H_0$ and $H$, we will prove
\begin{lemma}\label{lma.c3}
Let Assumption~\ref{ass1} hold true.
Then 
\begin{equation}
D_\eps^{(2)}
=
4\pi\Im (\bZ_0 P_- \bY_0 \bpsi_\eps \bchi_0 P_+ \bZ_0^*)
\quad
\text{ in } L^2(\bbR,\calN).
\label{c14}
\end{equation}
\end{lemma}

In \eqref{c14}, the projections $P_+$ and $P_-$ and the multiplication operators
$\bpsi_\eps$ and $\bchi_0$ act in $L^2(\bbR,\calK)$. 

\subsection{Swapping $\bZ_0$ and $P_\pm$: $D_\eps^{(2)}\to D_\eps^{(3)}$}\label{sec.c5}

Our next step is to swap $\bZ_0$ with $P_-$ and $P_+$ with $\bZ_0^*$ in 
the representation \eqref{c14}. Set
\begin{equation}
D_\eps^{(3)}
=
4\pi\Im (P_- \bpsi_\eps \bchi_0 \bZ_0 \bY_0 \bZ_0^*  P_+ )
\quad
\text{ in } L^2(\bbR,\calN).
\label{c15}
\end{equation}
In \eqref{c15}, $P_\pm$ act in $L^2(\bbR,\calN)$. 
We will prove
\begin{lemma}\label{lma.c4}
Let Assumption~\ref{ass1} hold true
and let the exponent $q$ satisfy $q\geq 2p$, $q>1/\gamma$.  
Then 
$$
\norm{D_\eps^{(2)}-D_\eps^{(3)}}_q=O(1), \quad \eps\to+0.
$$
\end{lemma}
The proof will be achieved by a straightforward application of the results 
of \cite[Section 6.9]{Peller}, where Schatten norm estimates for commutators of $P_\pm$ 
with operator-valued functions are given. 

Comparing \eqref{c15} with the definition \eqref{c6} of SHO,
we find that the operator $D_\eps^{(3)}$ is in fact a SHO in $L^2(\bbR,\calN)$:
$$
D_\eps^{(3)}=\SHO(\Omega_\eps^{(3)}), \quad
\Omega_\eps^{(3)}=-2\pi i \psi_\eps \chi_0 Z_0 Y_0 Z_0^*.
$$
This already shows (see the discussion after Lemma~\ref{lma.c1}) that $\Tr(D_\eps^{(3)})^m=0$ 
for all odd $m$, whenever the trace exists. 

It is important that we can rewrite the symbol $\Omega_\eps^{(3)}$ 
in terms of the scattering matrix $S(\lambda)$ 
for the pair $H_0$, $H$. 
Recall the stationary representation for $S(\lambda)$ (see e.g. \cite[Section 5.5]{Yafaev1}): 
\begin{equation}
S(\lambda)=I-2\pi i Z(\lambda)Y(\lambda+i0)Z(\lambda)^*, \quad \lambda\in \delta_0.
\label{c18}
\end{equation}
We denote 
$$
S_0(\lambda)=I-2\pi i Z_0(\lambda)Y_0(\lambda)Z_0(\lambda)^*, \quad \lambda\in \bbR.
$$
Thus, $S_0(\lambda)=S(\lambda)$ in a neighbourhood of $\lambda=0$ and by \eqref{c12},
$$
S_0-I\in C^\gamma(\bbR,\Sch_p).
$$
With this notation, we can rewrite the symbol $\Omega_\eps^{(3)}$ as 
$$
\Omega_\eps^{(3)}(\lambda)
=
\psi_\eps(\lambda)\chi_0(\lambda)(S_0(\lambda)-I), \quad \lambda\in\bbR. 
$$

\subsection{Replacing $S_0(\lambda)$ by $S(0)$: $D_\eps^{(3)}\to D_\eps^{(4)}$}\label{sec.c6}

At this step, we replace the symbol $\Omega^{(3)}_\eps$ with 
$$
\Omega_\eps^{(4)}(\lambda)
=
\psi_\eps(\lambda)\chi_0(\lambda)(S(0)-I), \quad \lambda\in\bbR. 
$$
Set $D_\eps^{(4)}=\SHO(\Omega^{(4)}_\eps)$; we prove
\begin{lemma}\label{lma.c5}
Let Assumption~\ref{ass1} hold true, and assume \eqref{a6aa}.
Let the exponent $q$ satisfy $q\geq 2p$, $q>1/\gamma$.  
Then 
$$
\norm{D_\eps^{(3)}-D_\eps^{(4)}}_q=O(1), \quad \eps\to+0.
$$
\end{lemma}
The proof is based on  the H\"older continuity of $S_0(\lambda)$ and 
again uses the estimates from \cite[Section 6.9]{Peller}. 

It is important that $\Omega_\eps^{(4)}(\lambda)$ is a scalar multiple
of a single operator $S(0)-I$ in $\calN$. 
Identifying $L^2(\bbR,\calN)$ with $L^2(\bbR)\otimes \calN$,
we can write the SHO with the symbol $\Omega_\eps^{(4)}$ as 
$$
\SHO(\Omega_\eps^{(4)})
=
(P_- 2\bpsi_\eps\bchi_0 P_+)\otimes \tfrac12 (S(0)-I)
+
(P_- 2\bpsi_\eps\bchi_0 P_+)^*\otimes \tfrac12 (S(0)^*-I).
$$
Using Lemma~\ref{lma.c1}, we see that for even $m$ 
\begin{equation}
\Tr(\SHO(\Omega_\eps^{(4)}))^m
=
2\Tr\abs{P_- \bOmega_\eps^{(4)} P_+}^m
=
2\Tr \abs{P_- 2 \bpsi_\eps \bchi_0 P_+}^m 
\sum_{n=1}^N a_n(0)^m
\label{c25}
\end{equation}
(recall that $a_n(0)$ are the eigenvalues of $\tfrac12\abs{S(0)-I}$). 
Thus, the problem reduces to analysing the \emph{scalar} SHO with the 
symbol $2\psi_\eps\chi_0$.

\subsection{Replacing $2\psi_\eps\chi_0$ by a model symbol : $D_\eps^{(4)}\to D_\eps^{(5)}$}\label{sec.c7}
It turns out that the leading term of the asymptotics of the trace in the r.h.s. of  \eqref{c25} 
is independent of the details of the behaviour of the function $\psi(\lambda)$, 
as long as it converges sufficiently fast to the limits $\mp1/2$ as $\lambda\to\pm\infty$. 
Thus, we are going to replace the symbol $2\psi_\eps\chi_0$ by an explicit model symbol. 
Let 
\begin{equation}
\zeta(\lambda)=-\frac2\pi \tan^{-1}(\lambda),
\quad
\zeta_\eps(\lambda)=\zeta(\lambda/\eps), 
\quad
\lambda\in \bbR.
\label{c26}
\end{equation}
We set 
$$
D_\eps^{(5)}=\SHO(\Omega_\eps^{(5)}),
\quad
\Omega_\eps^{(5)}(\lambda)
=
(\zeta_\eps(\lambda)-\zeta(\lambda))\tfrac12(S(0)-I), 
\quad
\lambda\in\bbR
$$
and prove
\begin{lemma}\label{cr.c7}
Let Assumption~\ref{ass1} hold true, and assume \eqref{a6aa}.
Let $\zeta$ be given by \eqref{c26}. Then 
for $q=\max\{2,p\}$ one has
$$
\norm{D_\eps^{(4)}-D_\eps^{(5)}}_q=O(1), \quad \eps\to+0.
$$
\end{lemma}
The proof is based on an elementary Hibert-Schmidt estimate for 
scalar HOs.

\subsection{Computing $\Tr(D_\eps^{(5)})^m$}\label{sec.c8}

Just as in \eqref{c25}, for even $m$ we have
\begin{equation}
\Tr(D_\eps^{(5)})^m 
=
2\Tr\abs{P_- (\bzeta_\eps-\bzeta) P_+}^m 
\sum_{n=1}^N a_n (0)^m. 
\label{c30}
\end{equation}
It turns out that the operator $\abs{P_- (\bzeta_\eps-\bzeta) P_+}$ can be
explicitly identified and the asymptotics of its traces can be computed. 
Let $J$ be the involution in $L^2(\bbR)$ given by $(Jf)(x)=f(-x)$. 
Evidently, $J$ maps $H^2_+(\bbR)$ into $H^2_-(\bbR)$ and vice versa. 
Consider the operator
$$
K_\eps=-iP_+(\bzeta_\eps-\bzeta)JP_+
\quad
\text{ in $H^2_+(\bbR)$}
$$
for $0<\eps<1$. 
Since $\zeta$ is odd, we have $(\bzeta_\eps-\bzeta)J=-J(\bzeta_\eps-\bzeta)$
and so $K_\eps$ is self-adjoint. 
We prove 
\begin{lemma}\label{lma.c8}
For $\eps\in(0,1)$, we have $K_\eps=\abs{P_- (\bzeta_\eps-\bzeta) P_+}$. 
This operator belongs to the trace class and
for all integers $m\geq1$
\begin{equation}
\Tr K_\eps^m 
=
\abs{\log \eps}\frac1{2\pi^2}
\int_{-\infty}^\infty (\cosh x)^{-m} dx+O(1)
\quad 
\eps\to+0.
\label{c31}
\end{equation}
\end{lemma}
In the proof of this lemma, we use a calculation from \cite{FP}.
Of course, \eqref{c30} and \eqref{c31} give the statement $P^{(5)}$, see \eqref{c4}.

\section{Spectral localization}\label{sec.d}

Here we prove Lemmas~\ref{lma.c2} and \ref{lma.b0}. 
We start with a proposition which is essentially well known:

\begin{lemma}\label{lma.d1}
\begin{enumerate}[\rm (i)]
\item
Let Assumption~\ref{ass1}(1) hold true
and let $\f\in C(\bbR)$ be such that $\f(x)\to\const$ as $x\to+\infty$. 
Then
$$
\f(H)-\f(H_0)\in\Sch_\infty.
$$
In particular, under the assumption \eqref{a6} on $\psi$, 
the operator $D_\eps(\lambda)$ is compact 
for any $\lambda\in\bbR$ and $\eps>0$.
\item
Let Assumption~\ref{ass1}(1), (2) hold true
and let $\f\in C^\infty(\bbR)$ be such that $\f(x)=\const$ 
for all sufficiently large $x>0$. Then 
$$
\f(H)-\f(H_0)\in \Sch_p.
$$
In particular, under the assumption \eqref{a6aa} on $\psi$, 
the operator $D_\eps(\lambda)$ belongs to $\Sch_p$ 
for any $\lambda\in\bbR$ and $\eps>0$.
\end{enumerate}
\end{lemma}
(Item (i) proves Lemma~\ref{lma.b0}.)
\begin{proof}
(i) 
We first note that the values $\f(x)$ for $x<\min\{\inf\sigma(H),\inf\sigma(H_0)\}$ 
do not matter, and we may modify the definition of $\f$ for such $x$
as we wish. Thus, we can represent $\f$ as 
$\f(x)=\const+\f_0(x)$, where $\f_0(x)\to0$ as $\abs{x}\to\infty$. 
So it suffices to prove the statement of the Lemma
for the case $\f\in C(\bbR)$ with $\f(x)\to0$ as $\abs{x}\to\infty$.
Next, let $b>-\min\{\inf\sigma(H_0),\inf\sigma(H)\}$. By the resolvent
identity (see \eqref{e1}), we have
$$
R(-b)-R_0(-b)\in\Sch_\infty.
$$
It follows that
$$
\wt\f(R(-b))
-
\wt\f(R_0(-b))
\in\Sch_\infty
$$
for all polynomials $\wt \f$. By the Weierstrass approximation theorem,
the same is true for all continuous functions $\wt \f$. 
Now choosing $\wt\f$ such that $\wt\f(1/(x+b))=\f(x)$, we obtain 
$$
\f(H)-\f(H_0)
=
\wt\f(R(-b))
-
\wt\f(R_0(-b))
\in\Sch_\infty,
$$
as required. 

(ii)
Similarly to part (i), by subtracting a constant from $\f$ we reduce the situation to 
the case $\f\in C_0^\infty(\bbR)$. 
From the inclusion \eqref{b2a} it is not difficult to deduce that
\begin{equation}
\wt\f((H+bI)^{-k})
-
\wt\f((H_0+bI)^{-k})
\in\Sch_p, 
\quad
\forall \wt \f\in C_0^\infty(\bbR). 
\label{d3}
\end{equation}
The implication \eqref{b2a} $\Rightarrow$ \eqref{d3}
can be obtained by any of a number of standard methods.
For example, it follows from \cite[Theorem 10]{BS};
alternatively, one can use the functional calculus based on 
the almost analytic continuation of $\wt \f$, see e.g. 
\cite[Section 8]{DSj}. 
Now choosing $\wt\f$ such that $\wt\f(1/(x+b)^k)=\f(x)$, we obtain 
$$
\f(H)-\f(H_0)
=
\wt\f((H+bI)^{-k})
-
\wt\f((H_0+bI)^{-k})
\in\Sch_p,
$$
as required. 
\end{proof}

\begin{proof}[Proof of Lemma~\ref{lma.c2}]

Let $\chi_\infty=1-\chi_0$. 
First consider the product 
\begin{multline}
\chi_\infty(H_0)D_\eps
=
\chi_\infty(H_0)(\psi_\eps(H)-\psi_\eps(H_0))
\\
=
(\chi_\infty(H_0)-\chi_\infty(H))\psi_\eps(H)
+
(\chi_\infty(H)\psi_\eps(H)-\chi_\infty(H_0)\psi_\eps(H_0)).
\label{d4}
\end{multline}
For the first term in the r.h.s. here we get 
\begin{equation}
\norm{(\chi_\infty(H_0)-\chi_\infty(H))\psi_\eps(H)}_p
\leq
\norm{\chi_\infty(H_0)-\chi_\infty(H)}_p
\norm{\psi}_{L^\infty}
<\infty
\label{d5}
\end{equation}
by Lemma~\ref{lma.d1}(ii). 
Consider the second term in the r.h.s. of \eqref{d4}. 
Denote $\psi_0(x)=-\tfrac12\sign(x)$. 
By our assumption \eqref{a6aa} on $\psi$, we have
\begin{equation}
\chi_\infty\psi_\eps=\chi_\infty\psi_0
\quad
\text{ for all sufficiently small $\eps>0$,}
\label{d6}
\end{equation}
and therefore for such $\eps$ the second term in the r.h.s. of \eqref{d4} 
becomes
$$
\chi_\infty(H)\psi_0(H)-\chi_\infty(H_0)\psi_0(H_0),
$$
which is in $\Sch_p$
by Lemma~\ref{lma.d1}(ii). 
Together with \eqref{d5}, this yields
\begin{equation}
\norm{D_\eps-\chi_0(H_0)D_\eps}_p
=
\norm{\chi_\infty(H_0)D_\eps}_p
=
O(1)
\label{d7}
\end{equation}
as $\eps\to+0$. Similarly,
\begin{multline}
\norm{\chi_0(H_0)D_\eps-\chi_0(H_0)D_\eps\chi_0(H_0)}_p
=
\norm{\chi_0(H_0)D_\eps\chi_\infty(H_0)}_p
\\
\leq 
\norm{\chi_0}_{L^\infty}
\norm{D_\eps\chi_\infty(H_0)}_p
=
O(1)
\label{d8}
\end{multline}
as $\eps\to+0$. Finally, denote $\wt \chi_\infty=1-\chi_0^2$; 
we have, using \eqref{d6}, 
\begin{multline}
\chi_0(H_0)D_\eps\chi_0(H_0)-D_\eps^{(1)}
=
\chi_0(H_0)(\wt \chi_\infty(H)\psi_\eps(H)-\wt \chi_\infty(H_0)\psi_\eps(H_0))\chi_0(H_0)
\\
=
\chi_0(H_0)(\wt \chi_\infty(H)\psi_0(H)-\wt \chi_\infty(H_0)\psi_0(H_0))\chi_0(H_0)
\label{d9}
\end{multline}
for all sufficiently small $\eps$. 
By Lemma~\ref{lma.d1}(ii), the expression in brackets in the r.h.s. of \eqref{d9}
is in $\Sch_p$, and so we obtain
$$
\norm{\chi_0(H_0)D_\eps\chi_0(H_0)-D_\eps^{(1)}}_p=O(1)
$$
as $\eps\to+0$. 
Combining this with \eqref{d7} and \eqref{d8}, we obtain the claim of the lemma.
\end{proof}

\section{Application of the resolvent identity}\label{sec.e}

\begin{proof}[Proof of Lemma~\ref{lma.c3}]
In fact, this calculation has appeared before in \cite{Push2}. 
It is based on the iterated resolvent identity written in the form
\begin{equation}
R(z)-R_0(z)
=
-(GR_0(\overline{z}))^*Y(z)GR_0(z), 
\quad 
\Im z>0.
\label{e1}
\end{equation}
Let us recall the derivation of \eqref{e1} (see e.g. \cite[Section~1.9]{Yafaev1}). 
In order to avoid inessential technical explanations related to the
operator $G^*$, let us assume here that $G$ is bounded.
Iterating the usual resolvent identity, we get
\begin{align}
R(z)-R_0(z)
=
-R(z)VR_0(z)
&=
-R_0(z)VR_0(z)
+R_0(z)VR(z)VR_0(z)
\label{e1a}
\\
&=-R_0(z)G^*V_0(I-GR(z)G^*V_0)GR_0(z).
\label{e1b}
\end{align}
We also have the identity 
$$
(I-GR(z)G^*V_0)(I+GR_0(z)G^*V_0)=I,
$$
which can be verified by expanding and using \eqref{e1a}. 
Writing $(I+GR_0(z)G^*V_0)^{-1}$ instead of $(I-GR(z)G^*V_0)$ in  
\eqref{e1b} and recalling the definition \eqref{c10} of 
$Y(z)$, we obtain \eqref{e1}.

Let $\f\in C_0^\infty (\delta_0)$ (we will eventually take $\f=\psi_\eps\chi_0^2$). 
By a version of Stone's formula, for any $u\in\calH$ we have
\begin{equation}
(\f(H)u,u)
=
\frac1\pi
\lim_{\epsilon\to+0}
\Im \int_{-\infty}^\infty(R(x+i\epsilon)u,u)\f(x)dx.
\label{e2}
\end{equation}
Subtracting the analogous formula for $(\f(H_0)u,u)$ from \eqref{e2} 
and using \eqref{e1}, we obtain
\begin{multline*}
((\f(H)-\f(H_0))u,u)
\\
=
-\frac1\pi \lim_{\epsilon\to+0}
\Im\int_{-\infty}^\infty
(Y(x+i\epsilon)GR_0(x+i\epsilon)u,GR_0(x-i\epsilon)u)\f(x)dx.
\end{multline*}
Now let us apply this to $u=\chi_0(H_0)f$, where $f\in\Ran \1_{\delta_0}(H_0)$. 
By the strong smoothness assumption \eqref{b4}, we have
$$
GR_0(z)\chi_0(H_0)f
=
\int_{\delta_0}\frac{Z(t)^*\wt f(t)}{t-z}\chi_0(t)dt
=
\int_{\bbR}\frac{Z_0(t)^*\wt f(t)}{t-z}dt, 
\quad 
\wt f=\calF f.
$$
Combining these formulas, we obtain
\begin{multline*}
((\f(H)-\f(H_0))\chi_0(H_0)f,\chi_0(H_0)f)
\\
=
4\pi \lim_{\epsilon\to+0} \Im 
\int_\bbR dx \int_\bbR dt \int_\bbR ds
(M_\epsilon(x,t,s)\wt f(t),\wt f(s)),
\end{multline*}
where 
$$
M_\epsilon(x,t,s)
=
Z_0(s)\frac{1}{2\pi i}\frac{1}{s-x-i\epsilon} Y(x+i\epsilon)\f(x)\biggl(-\frac1{2\pi i}\biggr)
\frac1{x-t+i\epsilon}Z_0(t)^*.
$$
Now let us use this identity with $\f\chi_0$ instead of $\f$. 
Recalling the formulas \eqref{a11} for $P_\pm$, we obtain 
$$
\chi_0(H_0)
(\f(H)\chi_0(H)-\f(H_0)\chi_0(H_0))
\chi_0(H_0)
=
4\pi \Im (\bZ_0 P_- \bY_0 \pmb\f P_+ \bZ_0^*).
$$
Finally, we substitute $\f=\psi_\eps\chi_0$ to obtain the required identity.
\end{proof}

\section{Swapping $Z_0$ and $P_\pm$ }\label{sec.f}

We start by giving the following corollary of the general results of \cite[Section 6.9]{Peller}:

\begin{lemma}\label{lma.f1}
Let $\fh_1$, $\fh_2$ be Hilbert spaces and let $\Omega$ be a function on $\bbR$
with values in the set of compact operators acting from $\fh_1$ to $\fh_2$. 
Let $q>1$ and $\gamma>1/q$; 
assume that $\Omega\in C^\gamma(\bbR,\Sch_q)$ and $\supp\Omega\subset(-r,r)$. 
Then the operators
$$
P_\mp\bOmega P_\pm: L^2(\bbR,\fh_1)\to L^2(\bbR,\fh_2)
$$
belong to $\Sch_q$ with the norm bound
\begin{equation}
\norm{P_\mp \bOmega P_\pm}_q 
\leq
C(q,\gamma,r) \norm{\Omega}_{C^\gamma(\bbR,\Sch_q)}.
\label{f1}
\end{equation}
\end{lemma}
In the product $P_\mp\bOmega P_\pm$, the projection $P_\pm$ on the right acts in $L^2(\bbR,\fh_1)$ and 
the projection $P_\mp$ on the left acts in $L^2(\bbR,\fh_2)$. 
Notation $\Sch_q$ in the statement of the lemma is used in two different senses:
in $C^\gamma(\bbR,\Sch_q)$ this is the Schatten class of operators acting from $\fh_1$ to $\fh_2$
and in the line above \eqref{f1} it is the Schatten class of operators acting from 
$L^2(\bbR,\fh_1)$ to $L^2(\bbR,\fh_2)$. 

\begin{proof}
The relevant results in \cite{Peller} are stated for functions 
on the unit circle $\bbT$ rather than on the real line. In order to make the 
connection, consider the unitary operator
$U_j:L^2(\bbT,\fh_j)\to L^2(\bbR,\fh_j)$, $j=1,2$, corresponding to the standard 
conformal map from the unit circle to the real line:
$$
(U_jf)(x)=\frac1{\sqrt{\pi}}\frac1{x+i}f\left(\frac{x-i}{x+i}\right),
\quad x\in\bbR, 
\quad f\in L^2(\bbT,\fh_j).
$$
Then 
$$
U_j^* P_\pm U_j=p_\pm, 
$$
where $p_\pm$ are the orthogonal projections in $L^2(\bbT,\fh_j)$ 
onto the Hardy classes $H^2_\pm(\bbT,\fh_j)$,
and 
$$
U_2^* \bOmega U_1=\bomega,
$$
where
$\bomega$ is the operator of multiplication by the function $\omega$ on $\bbT$ 
obtained from $\Omega$ by the change of variable:
\begin{equation}
\omega(e^{i\theta})=\Omega\biggl(i\frac{1+e^{i\theta}}{1-e^{i\theta}}\biggr),
\quad 
e^{i\theta}\in\bbT.
\label{f2}
\end{equation}
Corollary 9.4 from \cite{Peller} states that $p_\mp\bomega p_\pm\in\Sch_q$ if and only if
$\bomega$ belongs to the operator-valued Besov class
$B^{1/q}_{q,q}(\bbT,\Sch_q)$. 
This class for $q>1$ is defined as the class of all
$\Sch_q$-valued functions $\omega$ on $\bbT$ such that the norm
given by 
\begin{equation}
\norm{\omega}_{B^{1/q}_{q,q}}^q
=
\int_{-\pi}^\pi
\int_{-\pi}^\pi
\frac{\norm{\omega(e^{i(\theta+\tau)})-\omega(e^{i\theta})}_q^q}{\abs{e^{i\tau}-1}^2}
d\theta d\tau
\label{f3}
\end{equation}
is finite. 
(In fact, the norm \eqref{f3} vanishes on constant functions, so the precise 
definition of this Besov class involves taking a quotient over constants.)
One also has the corresponding norm bound
$$
\norm{p_\mp \bomega p_\pm}_q
\leq 
C_q \norm{\omega}_{B^{1/q}_{q,q}}.
$$
By our assumption and by \eqref{f2}, we have
$\omega\in C^\gamma(\bbT,\Sch_q)$:
$$
\sup_\theta\norm{\omega(e^{i(\theta+\tau)})-\omega(e^{i\theta})}_q
\leq
C\abs{e^{i\tau}-1}^\gamma.
$$
Thus, the Besov norm in \eqref{f3} is finite if $q\gamma>1$,
and we obtain \eqref{f1}.
\end{proof}

\begin{proof}[Proof of Lemma~\ref{lma.c4}]
We have
\begin{multline}
D_\eps^{(2)}-D_\eps^{(3)}
=
4\pi\Im((\bZ_0 P_--P_-\bZ_0)\bY_0\bpsi_\eps\bchi_0 P_+\bZ_0^*)
\\
+
4\pi\Im (P_- \bZ_0\bY_0\bpsi_\eps\bchi_0(P_+\bZ_0^*-\bZ_0^*P_+)).
\label{f4}
\end{multline}
Consider the first term in the r.h.s.; we have
\begin{equation}
\bZ_0 P_--P_-\bZ_0
=
P_+\bZ_0 P_- - P_-\bZ_0 P_+.
\label{f5}
\end{equation}
Since $Z_0\in C^\gamma(\bbR,\Sch_{2p})\subset C^\gamma(\bbR,\Sch_q)$
and $Z_0$ has a compact support, 
we can apply Lemma~\ref{lma.f1} to conclude that the operator
\eqref{f5} belongs to $\Sch_q$. Thus, we have
$$
\sup_{\eps>0}
\norm{(\bZ_0 P_--P_-\bZ_0)\bY_0\bpsi_\eps\bchi_0 P_+\bZ_0^*}_q
\leq
\norm{\bZ_0 P_--P_-\bZ_0}_q
\sup_{\eps>0}
\norm{\bY_0\bpsi_\eps\bchi_0 P_+\bZ_0^*}<\infty.
$$
Similar reasoning applies to the second term in the r.h.s.
of \eqref{f4}.
\end{proof}

\section{Replacing $S_0(\lambda)$ by $S(0)$ }\label{sec.g}

\begin{proof}[Proof of Lemma~\ref{lma.c5}]
We need to prove the estimate
$$
\norm{P_-(\mathbf S_0-S(0))\bpsi_\eps\bchi_0 P_+}_q=O(1), \quad \eps\to+0.
$$
By Lemma~\ref{lma.f1}, it suffices to prove
that
$$
\norm{(S_0-S(0))\psi_\eps\chi_0}_{C^\gamma(\bbR,\Sch_q)}=O(1), \quad \eps\to+0.
$$
In fact, we will prove a uniform estimate in $C^\gamma(\bbR,\Sch_p)\subset C^\gamma(\bbR,\Sch_q)$.
The proof of this estimate is an elementary argument which involves only the H\"older continuity 
of $S_0$ and the definition of $\psi$. 
Denote $\wt S(\lambda)=(S_0(\lambda)-S(0))\chi_0(\lambda)$. 
Let $\lambda_1, \lambda_2\in\bbR$ with $\abs{\lambda_1}\leq\abs{\lambda_2}$.
We need to prove the estimate
\begin{equation}
\norm{
\wt S(\lambda_2)\psi(\lambda_2/\eps)
-
\wt S(\lambda_1)\psi(\lambda_1/\eps)}_p
\leq 
C\abs{\lambda_1-\lambda_2}^\gamma
\label{g3}
\end{equation}
with $C$ independent of $\eps$. 
We have 
\begin{multline}
\wt S(\lambda_2)\psi(\lambda_2/\eps)
-
\wt S(\lambda_1)\psi(\lambda_1/\eps)
\\
=
(\wt S(\lambda_2)-\wt S(\lambda_1))\psi(\lambda_2/\eps)
+
\wt S(\lambda_1)(\psi(\lambda_2/\eps)-\psi(\lambda_1/\eps)).
\label{g2}
\end{multline}
For the first term in the r.h.s., by the H\"older continuity of $\wt S$ we immediately
obtain the required uniform bound:
$$
\norm{(\wt S(\lambda_2)-\wt S(\lambda_1))\psi(\lambda_2/\eps)}_p
\leq
\norm{\wt S(\lambda_2)-\wt S(\lambda_1)}_p
\norm{\psi}_{L^\infty}
\leq
C\abs{\lambda_2-\lambda_1}^\gamma.
$$
For the second term in the r.h.s. of \eqref{g2}, 
we have
\begin{multline*}
\norm{\wt S(\lambda_1)(\psi(\lambda_2/\eps)-\psi(\lambda_1/\eps))}_p
\leq
\norm{\wt S(\lambda_1)}_p
\abs{\psi(\lambda_2/\eps)-\psi(\lambda_1/\eps)}
\\
\leq
C\abs{\lambda_1}^\gamma
\abs{\psi(\lambda_2/\eps)-\psi(\lambda_1/\eps)}.
\end{multline*}
In order to estimate the expression in the  r.h.s., 
let us consider three cases:

\emph{Case 1: $\lambda_1\geq0$ and $\lambda_2\geq0$.}
Recall that we have also assumed $\abs{\lambda_1}\leq\abs{\lambda_2}$, so 
in this case we have $0\leq \lambda_1\leq \lambda_2$. 
Since $\psi(\lambda)=-1/2$ for $\lambda\geq R$, the case $\lambda_1\geq \eps R$ 
is trivial (the difference $\psi(\lambda_2/\eps)-\psi(\lambda_1/\eps)$ vanishes).
So let us assume that $\lambda_1< \eps R$. Then, by the H\"older continuity of $\psi$
(our function $\psi$ is $C^\infty$ smooth),
$$
\abs{\lambda_1}^\gamma \abs{\psi(\lambda_2/\eps)-\psi(\lambda_1/\eps)}
\leq
C\eps^\gamma R^\gamma\abs{\lambda_2/\eps-\lambda_1/\eps}^\gamma
\leq
CR^\gamma\abs{\lambda_2-\lambda_1}^\gamma.
$$

\emph{Case 2: $\lambda_1\leq0$ and $\lambda_2\leq0$.}
This case can be treated exactly as Case 1.

\emph{Case 3: $\lambda_1\lambda_2<0$.}
We have
$$
\abs{\lambda_1}^\gamma \abs{\psi(\lambda_2/\eps)-\psi(\lambda_1/\eps)}
\leq
2\norm{\psi}_{L^\infty}\abs{\lambda_1}^\gamma
\leq
2\norm{\psi}_{L^\infty}\abs{\lambda_2-\lambda_1}^\gamma.
$$
Thus, we have proven the uniform bound \eqref{g3}.
\end{proof}

\section{Replacing $2\psi_\eps\chi_0$ by a model symbol}\label{sec.h}

First let us prove an elementary estimate for scalar-valued symbols:
\begin{lemma}\label{lma.c6}
Let $\psi$ satisfy \eqref{a6aa} and let $\zeta$ be given by \eqref{c26}. Then 
$$
\norm{P_- (2\bpsi_\eps\bchi_0-(\bzeta_\eps-\bzeta))P_+}_2=O(1), \quad \eps\to+0.
$$
\end{lemma}
\begin{proof}
Denote, as in Section~\ref{sec.d}, 
$\chi_\infty=1-\chi_0$ and $\psi_0(\lambda)=-\tfrac12\sign(\lambda)$.
We have (using \eqref{d6})
\begin{equation}
2\psi_\eps\chi_0-(\zeta_\eps-\zeta)
=
(2\psi_\eps-\zeta_\eps)+(\zeta-2\psi_\eps\chi_\infty)
=
(2\psi_\eps-\zeta_\eps)+(\zeta-2\psi_0\chi_\infty)
\label{h1}
\end{equation}
for all sufficiently small $\eps$. 
First consider the second term in the r.h.s. of \eqref{h1}. 
Denote $\f=\zeta-2\psi_0\chi_0$; we have
\begin{equation}
P_-\pmb\f P_+=P_-(\pmb\f P_+ - P_+\pmb\f),
\label{h1a}
\end{equation}
and therefore it suffices to prove the inclusion
\begin{equation}
\pmb\f P_+ - P_+\pmb\f\in\Sch_2.
\label{h2}
\end{equation}
We have
\begin{equation}
\f\in C^\infty(\bbR), \quad
\f(x)=O(1/x), \quad \f'(x)=O(1/x^2), \quad \abs{x}\to\infty.
\label{h3}
\end{equation}
The integral kernel of $\pmb\f P_+ - P_+\pmb\f$
is 
\begin{equation}
-\frac1{2\pi i}\frac{\f(x)-\f(y)}{x-y}, \quad x,y\in \bbR.
\label{h4}
\end{equation}
It is an elementary calculation to check that 
conditions \eqref{h3} imply that the kernel \eqref{h4} 
belongs to $L^2(\bbR^2,dx dy)$. 
Thus, we obtain the inclusion \eqref{h2}. 

Next, consider the first term in the r.h.s. of \eqref{h1}.
Denote $\f=2\psi-\zeta$; 
by \eqref{h1a} it suffices to prove that the norm 
$\norm{\pmb \f_\eps P_+-P_+\pmb \f_\eps}_2$ is uniformly bounded. 
Again,  $\f$ 
satisfies \eqref{h3} and we have
$$
\int_\bbR\int_\bbR\Abs{\frac{\f(x/\eps)-\f(y/\eps)}{x-y}}^2 dxdy
=
\int_\bbR\int_\bbR\Abs{\frac{\f(x)-\f(y)}{x-y}}^2 dxdy,
$$
which is finite and independent of $\eps$. 
\end{proof}

\begin{proof}[Proof of Lemma~\ref{cr.c7}]
Since 
$$
D_\eps^{(4)}-D_\eps^{(5)}=\SHO(\Omega_\eps^{(4)}-\Omega_\eps^{(5)}),
$$
and
$$
\Omega_\eps^{(4)}-\Omega_\eps^{(5)}
=
\bigl(2\psi_\eps\chi_0-
(\zeta_\eps-\zeta )\bigr)
\tfrac12(S(0)-I), 
$$
we obtain (for $q=\max\{p,2\}$)
\begin{multline*}
\norm{D_\eps^{(4)}-D_\eps^{(5)}}_q
\leq
2\norm{P_-(\bOmega_\eps^{(4)}-\bOmega_\eps^{(5)})P_+}_q
\\
\leq
2\norm{P_-(2\bpsi_\eps\bchi_0-(\bzeta_\eps-\bzeta ))P_+}_q
\norm{\tfrac12(S(0)-I)}_q
\\
\leq
2\norm{P_-(2\bpsi_\eps\bchi_0-(\bzeta_\eps-\bzeta ))P_+}_2
\norm{\tfrac12(S(0)-I)}_p,
\end{multline*}
which is uniformly bounded by Lemma~\ref{lma.c6}.
\end{proof}

\section{Computing $\Tr(D_\eps^{(5)})^m$}\label{sec.i}
\begin{proof}[Proof of Lemma~\ref{lma.c8}]

\emph{Step 1:} let us prove that 
$\abs{P_-(\bzeta_\eps-\bzeta)P_+}=K_\eps$ for all $\eps\in(0,1)$ 
and that $K_\eps$ is trace class. 
We denote by $\Phi$ the standard unitary Fourier transform in $L^2(\bbR)$,
$$
(\Phi f)(t)=\frac1{\sqrt{2\pi}}\int_{-\infty}^\infty f(x)e^{-itx}dx.
$$
We have $\Phi (H^2_+(\bbR))=L^2(\bbR_+)$.
Denote 
$$
\wh K_\eps=\Phi K_\eps \Phi^* 
\quad\text{ in $L^2(\bbR_+)$.}
$$
The operator $\wh K_\eps$ is an integral operator with the kernel $k_\eps(t+s)$, 
$t,s\in\bbR_+$, where 
$$
k_\eps(t)=-\frac{i}{2\pi}\int_{-\infty}^\infty(\zeta_\eps(x)-\zeta(x))e^{-ixt}dx, 
\quad 
t>0.
$$
Recalling the explicit formula  \eqref{c26} for $\zeta$ and 
integrating by parts, we obtain
\begin{multline*}
k_\eps(t)
=
\frac{i}{\pi^2}\int_{-\infty}^\infty 
(\tan^{-1}(x/\eps)-\tan^{-1}(x))
e^{-ixt}dx
\\
=
-\frac{1}{\pi^2 t}\int_{-\infty}^\infty 
(\tan^{-1}(x/\eps)-\tan^{-1}(x))
(\tfrac{d}{dx}e^{-ixt})dx
\\
=
\frac{1}{\pi^2 t}\int_{-\infty}^\infty 
\left(\frac{1/\eps}{1+x^2/\eps^2}-\frac{1}{1+x^2}\right)e^{-ixt}dx
=
\frac{e^{-\eps t}-e^{-t}}{\pi t}.
\end{multline*}
Observe that the operator with this kernel 
 may be represented as
\begin{equation}
\wh K_\eps
=
\tfrac1\pi \calL \1_{(\eps,1)}\calL
=
\tfrac1\pi
(\1_{(\eps,1)}\calL)^*(\1_{(\eps,1)}\calL),
\label{i4}
\end{equation}
where $\calL$ is the Laplace transform in $L^2(\bbR_+)$, 
$$
(\calL f)(x)=\int_0^\infty e^{-tx}f(t)dt, 
$$
and $\1_{(\eps,1)}$ is the operator of multiplication by the characteristic function
of $(\eps,1)$ in $L^2(\bbR_+)$.
From \eqref{i4} it follows that $\wh K_\eps\geq0$ and therefore $K_\eps\geq0$.
In particular, this means that $\abs{K_\eps}=\sqrt{K_\eps^*K_\eps}=K_\eps$.
Since 
$$
P_-(\bzeta_\eps-\bzeta)P_+
=
JP_+J(\bzeta_\eps-\bzeta)P_+
=
-JP_+(\bzeta_\eps-\bzeta)JP_+
=-iJK_\eps,
$$
it follows that 
$$
\abs{P_-(\bzeta_\eps-\bzeta)P_+}=\abs{-iJK_\eps}
=
\sqrt{(-iJK_\eps)^*(-iJK_\eps)}=\abs{K_\eps}=K_\eps.
$$
Finally, by inspection, $\1_{(\eps,1)}\calL$ is a Hilbert-Schmidt operator
and therefore $\wh K_\eps$ is trace class. 

\emph{Step 2:}
We need to study the asymptotics of the traces
$\Tr K_\eps^m=\Tr\wh K_\eps^m$ and to prove that formula \eqref{c31} 
holds true for all natural $m$. 
This has been done in \cite{FP}; let us briefly recall the key steps of this argument. 
By \eqref{i4} and by the cyclicity of trace, we obtain
$$
\Tr {\wh K_\eps}^m
=
\Tr( \1_{(\eps,1)}(\tfrac1\pi\calL^{2}) \1_{(\eps,1)})^m.
$$
Next, we may apply the
the result of \cite{LaSa}:
\begin{equation}
\abs{
\Tr g( \1_{(\eps,1)}(\tfrac1\pi \calL^{2}) \1_{(\eps,1)})
-
\Tr \1_{(\eps,1)}g(\tfrac1\pi\calL^2)\1_{(\eps,1)}}
\leq
\frac12\norm{g''}_{L^\infty(0,1)}\norm{[\tfrac1\pi\calL^2,\1_{(\eps,1)}]}_2^2
\label{i5}
\end{equation}
for any $g\in C^2$ with $g(0)=0$; this result only uses the fact that 
$\1_{(\eps,1)}$ is an orthogonal projection and $\tfrac1\pi\calL^2$ is a bounded operator 
with the spectrum on the interval $[0,1]$.   We use \eqref{i5} this with $g(t)=t^m$. 
A direct calculation shows that the r.h.s. in \eqref{i5} is bounded as $\eps\to+0$.

Next, observe that $\calL^2$ is the Carleman operator, i.e. the 
integral operator in $L^2(\bbR_+)$ with the integral kernel 
$1/(s+t)$, $s,t\in\bbR$. Using the 
diagonalisation of the Carleman operator, one can compute the power
$(\tfrac1\pi\calL^2)^m$, which yields \cite{FP}
$$
\Tr (\1_{(\eps,1)}(\tfrac1\pi\calL^2)^{m} \1_{(\eps,1)})
=
\abs{\log \eps}\frac1{2\pi^2}\int_{-\infty}^\infty(\cosh x)^{-m}dx
+
O(1), \quad \eps\to+0
$$
for all $m\in\bbN$.
This gives the required result. 
\end{proof}

\section{Putting it all together: proof of Theorem~\ref{thm.b1}}\label{sec.j}

Fix an even integer $q$ such that $q\geq 2p$, $q>1/\gamma$. 

\emph{Step 1: assume \eqref{a6aa} and let  
$g(t)=t^m$, $m\in\bbN$, $m\geq q$.}

We follow the strategy outlined in 
Section~\ref{sec.c1}. By Lemmas~\ref{lma.c2}, \ref{lma.c4}, \ref{lma.c5} and \ref{cr.c7}, 
we have
\begin{equation}
\norm{D_\eps^{(i)}-D_\eps^{(i-1)}}_q=O(1), \quad \eps\to+0
\label{j1}
\end{equation}
for $i=1,3,4,5$. 
Recall that our aim is to prove the implication 
$$
P^{(i)} + \eqref{j1} \Rightarrow P^{(i-1)}
$$
for $i=1,3,4,5$, where $P^{(i)}$ is the statement
$$
P^{(i)}\!\!: \qquad 
\abs{\log\eps}^{-1}\Tr(D_\eps^{(i)})^m\to\Delta_m, 
\quad \eps\to+0, 
\quad \forall m\geq q.
$$
It suffices to prove that $P^{(i)}$ and \eqref{j1} imply
$$
\Tr(D_\eps^{(i)})^m
-
\Tr(D_\eps^{(i-1)})^m
=o(\abs{\log\eps}), \quad \eps\to+0.
$$
We use the following simple operator theoretic estimates:

(a) If $X\in\Sch_q$, $q\geq1$, then for all $m\geq q$
\begin{equation}
\norm{X}_m^m
=
\norm{\abs{X}^m}_1
\leq
\norm{X}^{m-q}\norm{\abs{X}^q}_1
=
\norm{X}^{m-q}\norm{X}_q^q.
\label{j3}
\end{equation}

(b) If $X,Y\in\Sch_m$, then 
\begin{equation}
\abs{\Tr X^m-\Tr Y^m}
\leq
m \norm{X-Y}_m\max\{\norm{X}_m^{m-1},\norm{Y}_m^{m-1}\}.
\label{j4}
\end{equation}
To prove \eqref{j4}, observe that  by cyclicity of the trace one has
$$
\Tr X^m - \Tr {Y}^m 
= 
\Tr (X-{Y})(X^{m-1} 
+ 
X^{m-2} {Y} + \cdots + X {Y}^{m-2} + {Y}^{m-1}),
$$
and so \eqref{j4} follows  by the application of the H\"older 
and the triangle inequality for Schatten classes.

We will take $X=D_\eps^{(i)}$, $Y=D_\eps^{(i-1)}$. 
By construction, we have a uniform bound of the operator norms:
\begin{equation}
\norm{D_\eps^{(i)}}=O(1), \quad \eps\to+0
\label{j10}
\end{equation}
for all $i=1,2,3,4,5$.
Next,  $P^{(i)}$ with $m=q$ gives (using that $q$ is an even integer)
$$
\norm{D_\eps^{(i)}}_q^q
=
\Tr(D_\eps^{(i)})^q
=
O(\abs{\log\eps}), \quad \eps\to+0.
$$
From here by \eqref{j3} we obtain 
$\norm{D_\eps^{(i)}}_m^m=O(\abs{\log\eps})$ for all $m\geq q$ and so
\begin{equation}
\norm{D_\eps^{(i)}}_m
=
O(\abs{\log\eps}^{1/m}), 
\quad \eps\to+0,
\quad \forall m\geq q.
\label{j5}
\end{equation}
By \eqref{j1}, \eqref{j3} and \eqref{j10} we get 
\begin{equation}
\norm{D_\eps^{(i)}-D_\eps^{(i-1)}}_{m}^m
\leq
\norm{D_\eps^{(i)}-D_\eps^{(i-1)}}^{m-q}
\norm{D_\eps^{(i)}-D_\eps^{(i-1)}}_q^{q}
=
O(1), \quad \eps\to+0.
\label{j6}
\end{equation}
Combining \eqref{j5} with \eqref{j6}, we get
\begin{equation}
\norm{D_\eps^{(i-1)}}_m
\leq
\norm{D_\eps^{(i)}}_m
+
\norm{D_\eps^{(i-1)}-D_\eps^{(i)}}_{m}
=O(\abs{\log\eps}^{1/m}), \quad \eps\to+0.
\label{j7}
\end{equation}
Substituting \eqref{j5}, \eqref{j6} and \eqref{j7} into \eqref{j4}, we get
$$
\abs{\Tr(D_\eps^{(i)})^m
-
\Tr(D_\eps^{(i-1)})^m}
\leq
O(\abs{\log\eps}^{(m-1)/m})
=o(\abs{\log\eps}), \quad \eps\to+0,
$$
as required. 

\emph{Step 2: Let $g$ be a polynomial with $g(t)=O(t^q)$ as $t\to0$.}
Write $g(t)=\sum_m g_m t^m$; by the previous step, we obtain
$$
\lim_{\eps\to+0}\abs{\log \eps}^{-1}\Tr g(D_\eps)
=
\sum_m g_m \Delta_m 
=
\frac1{\pi^2}\sum_{m\text{ even}} g_m \sum_{n=1}^N a_n(0)^m \int_{-\infty}^\infty (\cosh x)^{-m}dx.
$$
By the change of variable $y=a/\cosh x$, 
\begin{multline*}
a^m
\int_{-\infty}^\infty (\cosh x)^{-m}dx
=
2a^m\int_{0}^\infty (\cosh x)^{-m}dx
\\
=
2\int_0^a \frac{y^m}{y\sqrt{1-y^2/a^2}}dy
=
\int_{-a}^a \frac{y^m}{\abs{y}\sqrt{1-y^2/a^2}}dy
\end{multline*}
for even $m$, and so we obtain
$$
\lim_{\eps\to+0}\abs{\log \eps}^{-1}\Tr g(D_\eps)
=
\int_{-1}^1 g(y) \mu_0(y)dy,
$$
where $\mu_0$ is the weight defined by \eqref{a9} with $\lambda=0$.

\emph{Step 3: 
assume \eqref{a6aa} and let  $g=\1_\omega$, where $\omega\subset\bbR$ is an open interval 
such that $0\notin\overline{\omega}$.}
Let $A=2\norm{\psi}_{L^\infty}$; then $\norm{D_\eps}\leq A$. 
Let $g_\pm$ be polynomials with $g_\pm(t)=O(t^q)$, $t\to0$, and 
$$
g_-(t)\leq \1_\omega(t)\leq g_+(t), \quad \abs{t}\leq A.
$$
Then 
$$
\Tr g_-(D_\eps)\leq \Tr \1_\omega(D_\eps)\leq \Tr g_+(D_\eps).
$$
By the previous step, it follows that 
\begin{gather*}
\limsup_{\eps\to+0}\abs{\log \eps}^{-1}\Tr \1_\omega(D_\eps)
\leq
\lim_{\eps\to+0}\abs{\log \eps}^{-1}\Tr g_+(D_\eps)
=
\int_{-1}^1 g_+(t)\mu_0(t)dt,
\\
\liminf_{\eps\to+0}\abs{\log \eps}^{-1}\Tr \1_\omega(D_\eps)
\geq
\lim_{\eps\to+0}\abs{\log \eps}^{-1}\Tr g_-(D_\eps)
=
\int_{-1}^1 g_-(t)\mu_0(t)dt.
\end{gather*}
Taking infimum over all possible polynomials $g_+$, 
supremum over all possible polynomials $g_-$ and using
the Weierstrass' approximation theorem, we obtain the 
required asymptotic relation \eqref{b7} for $g=\1_\omega$.

\emph{Step 4: the general case.}

Let $g=\1_\omega$, where $\omega$ is as above, and now we suppose that 
$\psi$ satisfies \eqref{a6} instead of the stronger condition \eqref{a6aa}.
It suffices to consider the cases $\omega=(a,\infty)$ and $\omega=(-\infty,-a)$ with 
$a>0$. We consider the first case; the second one can be treated in the same way. 
Given any $d\in(0,a)$, let us represent $\psi=\psi^{(0)}+\psi^{(1)}$, where
$\psi^{(0)}$ satisfies the stronger condition \eqref{a6aa}, and $\norm{\psi^{(1)}}_{L^\infty}<d/2$.
Then 
$$
D_\eps
=
(\psi_\eps^{(0)}(H)-\psi_\eps^{(0)}(H_0))
-
(\psi_\eps^{(1)}(H)-\psi_\eps^{(1)}(H_0)),
$$
where $\norm{\psi_\eps^{(1)}(H)-\psi_\eps^{(1)}(H_0)}<d$, and therefore
$$
\psi_\eps^{(0)}(H)-\psi_\eps^{(0)}(H_0)-d\cdot I
\leq
D_\eps
\leq
\psi_\eps^{(0)}(H)-\psi_\eps^{(0)}(H_0)+d\cdot I.
$$
By the min-max, we have
$$
\Tr\1_{(a+d,\infty)}(\psi_\eps^{(0)}(H)-\psi_\eps^{(0)}(H_0))
\leq
\Tr \1_{(a,\infty)}(D_\eps)
\leq
\Tr\1_{(a-d,\infty)}(\psi_\eps^{(0)}(H)-\psi_\eps^{(0)}(H_0)).
$$
Applying the previous step of the proof and subsequently letting 
$d\to0$, we arrive at the required result. 

The case of a continuous $g$ which vanishes near the origin
follows by approximating $g$ from above and from below
by step-functions.

\section{Some applications}\label{sec.x}
In this section we give some examples of application of Theorem~\ref{thm.b1} 
to Schr\"odinger operators.

\subsection{Zero background potential}\label{sec.x1}

Let 
$$
H_0=-\Delta, \quad
H=-\Delta+V
\quad
\text{ in }
\calH=L^2(\bbR^d), \quad d\geq1,
$$ 
where
the real-valued potential $V$ satisfies
\begin{equation}
\abs{V(x)}\leq C(1+\abs{x})^{-\rho}, \quad \rho>1.
\label{b9}
\end{equation}
Let $\calK=\calH$,  and let $G$ be the operator of multiplication by
$G(x)=(1+\abs{x})^{-\rho/2}$; then $V_0=V_0(x)(1+\abs{x})^{\rho}$.  
The statement below is essentially well known. 

\begin{proposition}\label{pr.b2}
Assume \eqref{b9}. 
Then for any finite interval $\delta=(\lambda_1,\lambda_2)$ with $0<\lambda_1<\lambda_2<\infty$, 
for any $p>\max\{d/\rho,(d-1)/(\rho-1)\}$ and for any $k\in\bbN$ with $2pk>d$, Assumption~\ref{ass1} holds true. 
Moreover, for all $\lambda\in\delta$ the equation \eqref{b10}
has no non-trivial solutions and so $\delta_0=\delta$. 
\end{proposition}

Thus, in this case Theorem~\ref{thm.b1} holds true with any $\lambda>0$.

\begin{proof}
The inclusion 
\eqref{b2} is well-known. 
Further, by the Kato--Seiler--Simon bound  \cite[Theorem 4.1]{Simon},
we have for $b>0$
$$
(1+\abs{x})^{-\rho}(-\Delta+bI)^{-k}\in\Sch_{p}, \quad \rho p>d, \quad 2kp>d.
$$
Repeating word for word the argument of \cite[Theorem XII.12]{RS3}
for the class $\Sch_p$ instead of $\Sch_1$, we obtain the inclusion \eqref{b2a}.

Let us recall the formulas for the diagonalization $\calF_0$ of $H_0$.
The fiber space $\calN_d$ is $\calN_1=\bbC^2$ if $d=1$ and 
$\calN_d=L^2(\bbS^{d-1})$ if $d\geq2$. 
One defines the operator 
$$
\calF_0: L^2(\bbR^d)\to L^2((0,\infty); \calN_d)
$$
by 
\begin{align}
(\calF_0u)(\lambda)
&=\frac1{\sqrt{2}}\lambda^{-\frac{1}{4}}
(\hat u(\sqrt{\lambda}),\hat u(-\sqrt{\lambda})), 
\quad
\lambda\in(0,\infty), \quad d=1,
\label{x2}
\\
(\calF_0u)(\lambda,\omega)
&=\frac1{\sqrt{2}}\lambda^{\frac{d-2}{4}}\hat u(\sqrt{\lambda}\omega), 
\quad
\lambda\in(0,\infty), \quad \omega\in\bbS^{d-1}, \quad d\geq2,
\label{x1}
\end{align}
where $\hat u$ is the standard (unitary) Fourier transform of $u$. 
It is easy to see that $\calF_0$ diagonalises $H_0$. 
From \eqref{x2}, \eqref{x1} it is easy to obtain explicit formulas for the operator 
$Z_0(\lambda):L^2(\bbR^d)\to\calN_d$, defined by \eqref{b3a}. 
We have for any $\lambda>0$ 
\begin{align}
Z_0(\lambda)u
&=\frac1{\sqrt{2}}\lambda^{-\frac{1}{4}}
(\wh{Gu}(\sqrt{\lambda}),\wh{Gu}(-\sqrt{\lambda})), 
\quad d=1,
\label{x4}
\\
(Z_0(\lambda)u)(\omega)
&=\frac1{\sqrt{2}}\lambda^{\frac{d-2}{4}}\wh{Gu}(\sqrt{\lambda}\omega), 
\quad \omega\in\bbS^{d-1}, \quad d\geq2.
\notag
\end{align}
Thus, for $d=1$ the operator $Z_0(\lambda)$  has rank $\leq2$ and it is straightforward
to check that $Z_0(\lambda)$ is  H\"older continuous in $\lambda>0$. For $d\geq2$, it is also 
easy to prove 
the H\"older continuity of $Z_0(\lambda)$ in $\Sch_{2p}$ norm, $p>\frac{d-1}{\rho-1}$; 
this can be done by interpolating between the cases $p=1$ and $p=\infty$. 
See e.g. \cite[Lemma 8.1.8]{Yafaev2} for the details. 

Finally, the fact that equation \eqref{b10} has only trivial solutions 
for $\lambda>0$, follows by the argument 
involving Agmon's ``bootstrap'' and Kato's theorem on the absence of positive
eigenvalues of $H$; see e.g. \cite[Section 1.9]{Yafaev2} 
for an exposition of this argument.
\end{proof}

\subsection{Constant homogeneous magnetic field in three dimensions}\label{sec.x2}

Let us fix $B=\const>0$ and define the operator
\begin{equation}
H_0=(-i\nabla-A(x))^2-B \quad\text{ in $\calH=L^2(\bbR^3)$},
\label{x5}
\end{equation}
where
$$
A(x)=(-\tfrac12Bx_2,\tfrac12Bx_1,0)
$$
is the vector-potential corresponding to the constant homogeneous
magnetic field $(0,0,B)$ in $\bbR^3$. 
The term $-B$ in \eqref{x5} is added only for the purposes of normalization;
with this normalization, the spectrum of $H_0$ coincides with the interval $[0,\infty)$. 

The operator $H_0$ can be written as 
$H_0=h_0+(-\frac{\partial^2}{\partial x_3^2})$, where $h_0$ is the two-dimensional magnetic operator,
$$
h_0
=
\biggl(-i\frac{\partial}{\partial x_1}+\frac12 Bx_2\biggr)^2
+
\biggl(-i\frac{\partial}{\partial x_2}-\frac12 Bx_1\biggr)^2
-B
\quad\text{ in $L^2(\bbR^2)$.}
$$
The spectrum of $h_0$ consists of the Landau levels $\{0,2B,4B,\dots\}$,
which are the eigenvalues of $h_0$ of infinite multiplicity.
The Landau levels play the role of thresholds in the spectrum of $H_0$.
We set $H=H_0+V$, where the real-valued potential $V$ satisfies
\begin{equation}
\abs{V(x)}\leq 
C (1+\abs{x_1}+\abs{x_2})^{-\rho_\perp}
(1+\abs{x_3})^{-\rho}, 
\label{x6}
\end{equation}
with 
\begin{equation}
\rho>1 \quad\text{ and }\quad 0<\rho_\perp\leq\rho.
\label{x7}
\end{equation} 
Let $G$ be the operator of multiplication by 
$G(x)=G_\perp(x_1,x_2)G_3(x_3)$, where
$$
G_\perp(x_1,x_2)=(1+\abs{x_1}+\abs{x_2})^{-\rho_\perp/2},
\quad
G_3(x_3)=(1+\abs{x_3})^{-\rho/2}.
$$
\begin{proposition}\label{prp.x2}
Assume \eqref{x6}, \eqref{x7} 
and let $\delta\subset(0,\infty)$ be an open interval which does not contain any 
of the Landau levels $\{2nB\}_{n=0}^\infty$. Then Assumption~\ref{ass1} is satisfied
with $k=1$ and with any integer exponent $p>\max\{2,2/\rho_\perp\}$.
Let $\delta_*\subset\delta$ be the set of $\lambda\in\delta$ 
 where the equation \eqref{b10} has a non-trivial solution. 
 Then $\delta_*$ coincides with the set of eigenvalues of $H$ in $\delta$. 
\end{proposition}

We note that the set of the eigenvalues of $H$ in $\delta$ may be non-empty, see
examples in \cite[Section 5]{AHS}. 
In any case, we see that Theorem~\ref{thm.b1} holds true for any 
$$
\lambda>0, \quad \lambda\notin(\{2nB\}_{n=0}^\infty\cup \sigma_p(H)).
$$

\begin{proof}
(1)
The inclusion \eqref{b2} follows from the diamagnetic inequality,
see e.g. \cite[Section 2]{AHS}.

(2)
By our assumptions on $p$, we have $G\in L^{2p}(\bbR^3)$. 
By  the Kato--Seiler--Simon bound  \cite[Theorem 4.1]{Simon} it follows that 
$$
G(-\Delta+bI)^{-1}\in \Sch_{2p}, \quad b>0.
$$
Since $p$ is assumed to be an integer, by the diamagnetic inequality this implies
(see \cite[Section 2]{AHS}) that 
$$
G(H_0+bI)^{-1}\in \Sch_{2p}, \quad b>0.
$$
An application of the resolvent identity (see \eqref{e1b}) 
and of the Holder inequality for the Schatten classes (see \eqref{a14}) shows that
$$
(H+bI)^{-1}-(H_0+bI)^{-1}\in \Sch_p
$$
for $b>0$, $b>-\inf\sigma(H)$,
as required. 

(3)
As already mentioned, $H_0$ has a purely a.c. spectrum $[0,\infty)$ with multiplicity $N=\infty$.

(4)
Let us recall the explicit diagonalisation of $H_0$. 
Let $P_n$ be the orthogonal projection in $L^2(\bbR^2)$ onto the eigenspace
$\Ker (h_0-2nB)$ of $h_0$. 
Then 
$$
h_0=\sum_{n=0}^\infty 2nB P_n
$$
and so we get
$$
\varphi(H_0)=\sum_{n=0}^\infty P_n\otimes\varphi(-\tfrac{d^2}{dx_3^2}+2nB)
\quad
\text{ in $L^2(\bbR^3)=L^2(\bbR^2)\otimes L^2(\bbR)$}
$$
for any bounded function $\varphi$. 
The operator $-\frac{d^2}{dx_3^2}$ in $L^2(\bbR)$ is diagonalised by the map $\calF_0$, 
see \eqref{x2}. This allows one to diagonalise $H_0$. In order to display the 
corresponding formulas, let us assume for simplicity of notation that 
$\delta\subset(0,2B)$;
this corresponds to only one scattering channel being open on $\delta$. 
Then the diagonalisation operator $\calF$ for $H_0$ associated with the interval $\delta$ 
can be written as 
$$
\calF: \Ran \1_\delta(H_0)\mapsto L^2(\delta,\calN), 
\quad 
\calN=\Ran P_0\otimes\bbC^2,
$$
$$
\calF=P_0\otimes\calF_0.
$$
Further, 
let $Z_0(\lambda): L^2(\bbR)\to\bbC^2$ be the operator \eqref{x4} with $G=G_3$. 
Then the corresponding operator $Z(\lambda)$ for $H_0$ can be written as
$$
Z(\lambda): L^2(\bbR^3)\to\calN,\quad Z(\lambda)=P_0G_\perp\otimes Z_0(\lambda), \quad \lambda\in \delta.
$$
By our assumptions on the exponent $p$, we have 
$G_\perp\in L^{2p}(\bbR^2)$, and therefore $P_0G_\perp\in \Sch_{2p}$ 
(this can be proven by interpolating between the cases $p=2$ and $p=\infty$; see
e.g. \cite[Lemma 3.1]{FR}).
In the proof of Proposition~\ref{pr.b2} we have already seen that  
$Z_0(\lambda)$ is a rank two operator which is H\"older continuous in $\lambda>0$.
From here we obtain the strong $\Sch_{2p}$-smoothness assumption for $Z(\lambda)$. 

Finally, the required statement about $\delta_*$ was proven in \cite[Section 4]{AHS}.
\end{proof}

\subsection{Periodic operators}\label{sec.x3}
Let $H_0=-\Delta+W$ be the Schr\"odinger operator in $L^2(\bbR^d)$, $d\geq1$, with a periodic 
background potential $W$, and let $H=H_0+V$, where $V$ satisfies the estimate
\eqref{b9} with $\rho>1$. 
One can apply Theorem~\ref{thm.b1} to this pair of operators. 
In the one-dimensional case $d=1$ this is an easy exercise by using the 
standard facts about the diagonalisation of $H_0$, see e.g. \cite[Section XIII.16]{RS4}. 
In the multi-dimensional case, the verification of Assumption~\ref{ass1} 
meets the following difficulty: the local structure
of the energy band functions of the operator $H_0$ is not fully understood. 
As a consequence, the required facts about the diagonalisation operator $\calF$ are
not available for all energies. This issue can be overcome by making 
rather restrictive assumptions on the behaviour of the energy band functions, 
see e.g. \cite{BYa}.

\section*{Acknowledgements}
The author is grateful to Rupert Frank and Yuri Safarov for useful remarks.

\end{document}